\documentclass[10pt,twoside,a4paper,english,reqno]{amsart}
\usepackage{listings,graphicx,amsmath,varioref,amscd,amssymb,color,bm,stmaryrd}
\usepackage[pdftex,colorlinks=true,backref]{hyperref}

\numberwithin{equation}{section}
\allowdisplaybreaks

\newtheorem{theorem}{Theorem}[section]
\newtheorem{lemma}[theorem]{Lemma}

\theoremstyle{definition}
\newtheorem{definition}[theorem]{Definition}

\theoremstyle{remark}
\newtheorem{remark}[theorem]{Remark}

\newcommand{\Pth}{\partial_t^h}
\newcommand{\Div}{\operatorname{div}}
\newcommand{\Curl}{\operatorname{curl}}
\newcommand{\Grad}{\nabla}
\newcommand{\vr}{\varrho}
\newcommand{\vu}{\vc{u}}

\newcommand{\vc}[1]{{\bm{#1}}}

\newcommand{\weak}{\rightharpoonup}
\newcommand{\weakh}{\overset{h\to 0}{\weak}}
\newcommand{\toh}{\overset{h\to 0}{\longrightarrow}}

\newcommand{\dist}{\operatorname{dist}}

\newcommand{\norm}[1]{\left\Vert#1\right\Vert}
\newcommand{\abs}[1]{\left|#1\right|}
\newcommand{\Set}[1]{\left\{#1\right\}}
\newcommand{\jump}[1]{\left\llbracket #1\right\rrbracket}
\newcommand{\sjump}[1]{\left\llbracket #1\right\rrbracket_\Gamma}
\newcommand{\vrho}{\varrho}
\newcommand{\inb}{\in_{\text{b}}}
\newcommand{\Dt}{\Delta t}
\newcommand{\R}{\mathbb{R}}
\newcommand{\weakto}{\rightharpoonup}

\newcommand{\Om}{\ensuremath{\Omega}}
\newcommand{\cOm}{\ensuremath{\overline{\Omega}}}
\newcommand{\pOm}{\ensuremath{\partial\Omega}}
\newcommand{\Dom}{(0,T)\times\Omega}
\newcommand{\cDom}{[0,T)\times\overline{\Omega}}

\newcommand{\eff}{P_{\mathrm{eff}}}

\newcommand{\Aop}[1]{\mathcal{A}\left[#1\right]}
\newcommand{\Bop}[1]{\mathcal{B}\left[#1\right]}
\newcommand{\gradlaplaceinv}[1]{\mathcal{A}\left[#1\right]}

\newcommand{\avg}[1]{\left\langle #1 \right\rangle_\Om}

\newcommand{\solutiontext}
{
	Let $\Set{(\varrho_{h},\vc{u}_{h})}_{h>0}$ 
	be a sequence of numerical solutions 
	constructed according to \eqref{eq:num-scheme-II} 
	and Definition \ref {def:num-scheme}.
}

\title[A finite element method for compressible Stokes flow]{A convergent 
nonconforming finite element method for compressible Stokes flow}

\author[K. H. Karlsen]{Kenneth H. Karlsen}
\address[Kenneth H. Karlsen]{\newline
         Centre of Mathematics for Applications \newline
         University of Oslo\newline
         P.O. Box 1053, Blindern\newline
         N--0316 Oslo, Norway\newline
         and\newline
		 Center for Biomedical Computing,\newline
         Simula Research Laboratory\newline
         P.O. Box 134\newline
         N--1325 Lysaker, Norway}
\email[]{kennethk@math.uio.no}
\urladdr{http://folk.uio.no/kennethk}

\author[T. K. Karper]{Trygve K. Karper} 
\address[Trygve K. Karper]{\newline 
		Centre of Mathematics for Applications \newline 
		University of Oslo\newline 
		P.O. Box 1053, Blindern\newline 
		N--0316 Oslo, Norway}
\email[]{t.k.karper@cma.uio.no}
\urladdr{http://folk.uio.no/trygvekk/}

\date{\today}


\thanks{This work was supported by the Research Council of Norway through
an Outstanding Young Investigators Award (K. H. Karlsen). 
This article was written as part of the the international research program
on Nonlinear Partial Differential Equations at the Centre for
Advanced Study at the Norwegian Academy of Science
and Letters in Oslo during the academic year 2008--09.}

\begin{document}

\maketitle

\begin{abstract}
We propose a nonconforming finite element method for
isentropic viscous gas flow in situations 
where convective effects may be neglected. 
We approximate the continuity equation by a 
piecewise constant discontinuous Galerkin method. 
The velocity (momentum) equation is approximated by a
finite element method on div--curl form 
using the nonconforming Crouzeix--Raviart space.
Our main result is that the finite element method converges
to a weak solution. The main challenge is to demonstrate 
the strong convergence of the density approximations, which
is mandatory in view of the nonlinear pressure function.
The analysis makes use of a higher integrability estimate on
the density approximations, an equation for 
the ``effective viscous flux", and renormalized 
versions of the discontinuous Galerkin method.
\end{abstract}

\tableofcontents

\section{Introduction}
Let $\Om \subset \mathbb{R}^N$, with $N=2$ or $3$, be a bounded polygonal domain 
with Lipschitz boundary $\partial \Om$ and let $T>0$ be a fixed final time. 
In this paper, we consider the mixed hyperbolic-elliptic type system
\begin{align}
	\partial_t\varrho + \Div (\varrho\vc{u}) 
	&= 0, \quad \text{in $(0,T) \times \Omega$}, \label{eq:contequation}\\
	-\mu \Delta \vc{u} - \lambda \Grad \Div\vc{u} + \Grad p(\varrho) &= \vc{f}, 
	\quad \text{in $(0,T)\times \Omega$},
	\label{eq:momentumeq} 
\end{align}
with initial data
\begin{align}\label{initial}
	\varrho|_{t=0} & = \varrho_{0},
	\quad \textrm{in $\Omega$}.
\end{align}
The unknowns are the density $\varrho = \varrho(t,\vc{x}) \geq 0$ 
and the velocity $\vc{u} = \vc{u}(t,\vc{x}) \in \R^N$, with $\vc{x} \in \Omega$ and $t \in (0,T)$. 
The source term $\vc{f}$ is a given function representing body forces such 
as gravity. We denote by $\Div$ and $\Grad$ the usual spatial divergence 
and gradient operators and by $\Delta$ the Laplace operator. 
At the boundary $\partial \Om$, the system 
is supplemented with the homogenous Dirichlet condition
\begin{equation*}
	\vc{u} = 0, \quad \textrm{on $(0,T) \times \partial \Om$.}
\end{equation*}

The pressure $p(\varrho)$ is governed by the equation of state 
$p(\varrho) = a\varrho^\gamma$, $a>0$. Typical values of $\gamma$ ranges 
from a maximum of $\frac{5}{3}$ for monoatomic gases,  
through $\frac{7}{5}$ for diatomic gases \emph{including air}, to lower values close to $1$ for 
polyatomic gases at high temperatures. Throughout this paper, we will always 
assume that $\gamma>1$, which is the most difficult case. 
The viscosity coefficients $\mu, \lambda$ are assumed to be constant and 
satisfy $\mu> 0$, $N\lambda + 2 \mu \geq 0$.

The system \eqref{eq:contequation}--\eqref{eq:momentumeq} is a gross
simplification of the isentropic compressible Navier--Stokes equations. 
It provides a reasonable approximation in situations where convective effects may be neglected. 
Solutions of \eqref{eq:contequation}--\eqref{eq:momentumeq} have also been 
utilized by Lions \cite{Lions:1998ga} to construct 
solutions of the isentropic compressible Navier--Stokes equations.  
Regarding the mathematical theory, the semi--stationary system \eqref{eq:contequation}--\eqref{initial} 
has been analyzed by Lions \cite[Section 8.2]{Lions:1998ga}, among many others. 
More precisely, he proves the existence of weak 
solutions and provide some uniqueness and higher regularity results. 

In the literature one can find a variety of numerical methods for 
the compressible Stokes and Navier--Stokes equations. 
However, there are few results with reference to the convergence properties
of these methods, especially in several dimensions.
In one dimension, we refer to the works of Hoff 
and his collaborators \cite{Zarnowski:1991uq, Zhao:1994fk, Zhao:1997qy}. 
These results apply to the compressible Navier--Stokes equations written in Lagrangian 
form and requires the initial density to be of bounded variation. 
In several dimensions there are a few very recent results.  
In \cite{Gallouet1, Gallouet2}, the authors present a convergent finite element method 
for a Stokes model. This model is a stationary version of \eqref{eq:contequation}--\eqref{eq:momentumeq}. 
In their finite element method the approximation spaces for the density and velocity are the same.  
Moreover, their method is based on the standard weak formulation 
of the velocity equation \eqref{eq:momentumeq}. 
Since the finite element space is non-conforming, this approach may not preserve the div--curl 
structure of the continuous system. This complicates the convergence proof. 
In \cite{Gallouet1, Gallouet2},  additional stabilization terms are needed in the  
discretization of the continuity equation \eqref{eq:contequation}. 
In \cite{Karlsen1}, we construct a convergent mixed finite 
element method for \eqref{eq:contequation}--\eqref{eq:momentumeq}.
However, this method is based on a vorticity formulation 
of the velocity equation, which is only valid for the Navier slip boundary condition:
$$
\vu \cdot \nu  = 0, \quad \Curl \vu \times \nu = 0, \quad \text{ on }\pOm.
$$
In addition, the velocity is approximated by a $H(\Div)$ (Nedelec) element.

We now outline the numerical method proposed in this paper.  
First of all, the density $\vr$ is approximated by 
piecewise constants in the spatial and temporal variables.  
For the approximation of the velocity $\vc{u}$ we 
utilize the \emph{Crouzeix--Raviart} element 
space \cite{Crouzeix:1973qy} in the spatial 
variable, denoted by $\vc{V}_h(\Om)$, and 
piecewise constants in the temporal variable.
Hence, the numerical method is nonconforming in the 
sense that $\vc{V}_h \not \subset \vc{W}^{1,2}_0(\Om)$. 
In what follows, we mostly suppress the time variable $t$ and 
refer to subsequent sections for precise statements.
For the continuity equation \eqref{eq:contequation} we 
make use of a discontinuous Galerkin method. 
To achieve stability, the numerical fluxes are evaluated in the 
upwind direction dictated by the velocity. However, since the velocity 
space is not continuous across element faces, average 
velocities are used in this discretization. 
Our discontinuous Galerkin method is equivalent to a standard finite 
volume method for the continuity equation \cite{Feisteuer,Gallouet:2007lr}. 
In \cite{Karlsen1}, we use a similar discontinuous Galerkin method with the velocity 
in the div conforming Nedelec space of the first order and kind. 
Since the method used herein only depends on the average normal 
velocity at faces, the approximations constructed by this method are 
also solutions to the discrete continuity equation of \cite{Karlsen1}. 
More precisely, if the pair $(\vr_h,\vu_h)$ solves the discrete continuity 
equation proposed herein, then $(\vr_h,\Pi_h^N \vu_h)$ 
is a solution to the discrete continuity equation of \cite{Karlsen1}, where $\Pi_h^N$ 
is the canonical interpolation operator onto the div conforming Nedelec 
space of first order and kind.  As a consequence, several 
of the favorable properties of the method in \cite{Karlsen1} 
continue to hold for the continuity method herein. 
In particular, renormalized formulations, 
weak time-continuity, and consistency bounds are 
readily obtained by exploiting this connection.

To discretize the velocity equation \eqref{eq:momentumeq} we bring into service 
a non-standard finite element formulation, which starts off from the identity
\begin{equation}\label{eq:laplace}
	\int_{\Om}D\vc{u}D\vc{v}\ dx 
	= \int_{\Om}\Curl \vc{u}\Curl \vc{v} + \Div \vc{u}\Div \vc{v}\ dx, 
\end{equation}
valid for all $\vc{u} \in \vc{W}^{1,2}_0(\Om)$. 
However, since the velocity space is nonconforming, this identity does not hold
discretely, but we insist on utilizing the right-hand 
side of \eqref{eq:laplace} as a starting point for 
discretizing the velocity equation. 
Utilizing the form on the right--hand side,
it is possible to split the curl part of the 
Laplacian away from the divergence part.
By setting $\vc{v} = \Grad s$, we obtain the divergence part, while 
$\vc{v} = \Curl \vc{\eta}$ gives the curl part.
Of course, to satisfy boundary conditions, this 
argument must be localized. Discretely, this still 
holds for the element space $\vc{V}_h$
since this admits the exact orthogonal Hodge decomposition
$$
\vc{V}_h = \Curl \vc{\zeta}_h + \Grad S_h.
$$
Hence, the curl and divergence part of the Laplace operator can 
be separated by using test functions $\vc{v}_h = \Curl \vc{\zeta}_h$, $\vc{\zeta}_h \in \vc{W}_h$
or $\vc{v}_h = \Grad s_h, ~s_h \in S_h$. This property lies at the heart of 
the matter in the upcoming convergence analysis.

Contrasting with the standard situation in which the left--hand side of \eqref{eq:laplace}
is used, a discretization based on the right-hand 
side of \eqref{eq:laplace} does not converge unless additional terms
controlling the discontinuities of the velocity are added, cf.~Brenner \cite{Brenner}. 
The standard discretization of the Laplacian (based on the left--hand side
of \eqref{eq:laplace}) leads to a $L^2$ bound on $\Grad_h \vu_h$,
where $\Grad_h$ is the gradient restricted to each element $E$.
For the velocity space $\vc{V}_h$, this bound actually controls
the jump of $\vu_h$ across faces. This in turn, is sufficient 
to conclude that $\Grad_h \vu_h \weak \Grad \vu$ as $h \rightarrow 0$. 
When discretizing the Laplacian based on the right--hand side of \eqref{eq:laplace},
one obtains $L^2$ bounds on $\Curl_h \vu_h$ and $\Div_h \vu_h$, 
where $\Curl_h$ and $\Div_h$ denotes the curl and divergence 
operators, respectively, restricted to each element $E$.
The jump of $\vu_h$ across faces is not 
controlled by these terms. In fact, $\vc{V}_h$ contains 
non-zero functions for which both $\Div_h$ and $\Curl_h$ are 
zero. For this reason, extra terms controlling the jump
of $\vu_h$ across faces need to be added.

In choosing these terms we are inspired by 
the work of Brenner \cite{Brenner}, which deals with 
two-dimensional elliptic operators of 
the form ``$\Curl \Curl - \beta \Grad \Div$". 
To be more precise, our finite element method for the velocity equation
\eqref{eq:momentumeq} seeks $\vu_h \in \vc{V}_h(\Om)$ such that
\begin{equation}\label{intro:fem}
	\begin{split}
		&\int_{\Om} \mu\Curl_{h} \vc{u}_{h}\Curl_{h} \vc{v}_{h}
		+ \left[(\mu + \lambda)\Div_{h} \vc{u}_{h}
		-p(\varrho_{h})\right]\Div_{h} \vc{v}_{h}\ dx \\
		&\qquad\quad 
		+\mu \sum_{\Gamma \in \Gamma^I_{h}} h^{\epsilon -1}
		\int_{\Gamma}\sjump{\vc{u}_{h}\cdot \nu}\sjump{\vc{v}_{h}\cdot \nu}
		+ \sjump{\vc{u}_{h}\times \nu}\sjump{\vc{v}_{h}\times \nu}\ dS(x) 
		\\ & \quad
		= \int_{\Omega}\vc{f}_{h}\vc{v}_{h}\ dx, 
		\qquad \forall \vc{v}_{h}Ê\in \vc{V}_{h}(\Om),
	\end{split}
\end{equation}
for some fixed $\epsilon \in (0,1)$, where $\rho_h$ and $\vc{f}_h$ 
are given piecewise functions on $\Omega$ with respect 
to a tetrahedral mesh $E_h$ with elements $E$. 
Moreover, $\Gamma^I_h$ denote the set of internal faces, 
and $\jump{\cdot}_\Gamma$ denotes the jump across a face $\Gamma\in\Gamma^I_h$. 
The scaling factor $h^{\epsilon}$ is required 
to prove convergence of the finite element method. 
Of course, the size of $\epsilon$ will affect 
the accuracy of the method \cite{Brenner} 
and should be fixed very small in practical computations.

For any fixed $h>0$, let $(\vr_h, \vu_h) = (\vr_h, \vu_h)(t,x)$ 
denote the numerical solution to the compressible Stokes system. 
Our goal is to prove that $\Set{(\vrho_{h},\vc{u}_{h})}_{h>0}$ 
converges along a subsequence to a weak solution. 
The main challenge is to show that the density approximations 
$\vrho_h$, which a priori is only weakly compact in $L^2$, in fact converges strongly. 
Strong convergence is needed when sending $h \rightarrow 0$ in the nonlinear pressure function.
It is this issue that motivates the above nonconforming finite element method.
Since the finite element space $\vc{V}_h$ is piecewise linear and 
totally determined by its value at the faces, Green's theorem yield
$$
\Div_h \Pi_h^V \vc{v} = \Pi_h^Q \Div \vc{v}, \qquad \Curl_h \Pi_h^V\vc{v} = \Pi_h^Q \Curl \vc{v},
$$
where $\Pi_h^V$ is the canonical interpolation operator onto $\vc{V}_h$ and
$\Pi_h^Q$ is the $L^2$ projection onto piecewise constants. 
Consequently, the projection of a divergence or curl free function 
is again (piecewise) divergence or curl free. 
Using this, we see that the function $\vc{v}_h = \Pi_h^V \Grad \Delta^{-1} \vr_h$
is a solution to the div--curl problem
\begin{equation*}
	\Div_h \vc{v}_h = \vr_h, \quad \Curl_h \vc{v}_h = 0,	
\end{equation*}
away from the boundary. 
By using $\vc{v}_h$ as test function in \eqref{intro:fem}, the curl term vanishes, 
while the remaining terms constitute the so-called effective viscous flux 
$\eff(\vr_h, \vu_h)= p(\vr_h) - (\lambda + \mu)\Div \vu_h$, 
the source term, and the jump terms. The latter terms are shown to 
converge to zero. Using this, we are able to prove following weak continuity property:
\begin{equation}\label{eq:intro-weakcont}
	\lim_{h \rightarrow 0}\iint \eff(\vrho_{h},\vc{u}_{h})\,\vrho_{h}\phi ~dxdt 
	=\iint \overline{\eff}\, \vrho\phi ~dxdt
	\quad \text{($\overline{\eff}, \vrho$ are weak $L^2$ limits),}
\end{equation}
for all $\phi \in C_0^\infty$. This is the main ingredient in the 
strong convergence proof for the density approximations $\vr_h$.
The argument is inspired by the work of Lions 
on the compressible Navier-Stokes equations, cf.~\cite{Lions:1998ga}.

If we instead of \eqref{intro:fem}, discretize the Laplacian based on the left--hand 
side of \eqref{eq:laplace}, then the above analysis becomes more involved. 
In particular, it seems difficult to establish the key property \eqref{eq:intro-weakcont}.
In this case, we would need to establish 
$$
\int \int \Grad_h \vu_h \Grad_h \Pi_h^V 
\left[\Grad \Delta^{-1} \vr_h\right] - \Div_h \vu_h \vr_h~dxdt \rightarrow 0, 
\quad \text{as }h \rightarrow 0,
$$
which is intricate since all the involved quantities are only weakly convergent. 

The remaining part of this paper is organized as follows: 
In Section \ref{sec:prelim}, we first introduce some
relevant notation and state a few basic results from analysis. 
Next, we formulate our notion of a weak solution. 
Finally, we introduce the finite element spaces and derive 
some of their basic properties. 
In Section \ref{sec:numerical-method}, we present the 
numerical method and state our main convergence result. 
This section also provides a result regarding the 
existence of solutions to the discrete equations.
In Section \ref{sec:basic}, we derive stability and higher integrability
results. Section \ref{sec:conv} is devoted to proving the convergence 
result stated in Section \ref{sec:numerical-method}.

\section{Preliminary material}\label{sec:prelim}

\subsection{Functional spaces and analysis results}
We denote the spatial divergence and curl operators 
by $\Div$ and $\Curl$, respectively. As usual in the two dimensions, we 
denote both the rotation operator taking scalars into vectors 
and the curl operator taking vectors into scalars by $\Curl$.  

We will make use of the spaces
\begin{align*}
	\vc{W}^{\Div,2}(\Omega) & = 
	\Set{\vc{v} \in \vc{L}^2(\Omega): \Div \vc{v} \in L^2(\Omega)}, \\
	\vc{W}^{\Curl,2}(\Omega) & = 
	\Set{\vc{v} \in \vc{L}^2(\Omega): \Curl \vc{v} \in \vc{L}^2(\Omega)},
\end{align*}
where $\nu$ denotes the unit outward pointing normal vector on $\partial \Omega$. 
If $\vc{v}\in \vc{W}^{\Div, 2}(\Omega)$ satisfies $\vc{v} \cdot \nu|_{\partial \Omega}=0$, we 
write $\vc{v}\in \vc{W}^{\Div, 2}_{0}(\Omega)$. Similarly, $\vc{v}\in \vc{W}^{\Curl, 2}_0(\Omega)$ 
means $\vc{v}\in\vc{W}^{\Curl, 2}(\Omega)$ and $\vc{v} \times \nu|_{\partial \Omega} = 0$. 
From \cite{Girault:1986fu},
\begin{equation*}
	\vc{W}^{1,2}_0(\Om) = \vc{W}^{\Curl,2}_0 \cap \vc{W}^{\Div,2}_0.
\end{equation*}

The next lemma lists some basic results from functional analysis 
to be used in subsequent arguments (for proofs, see, e.g., \cite{Feireisl:2004oe}). 
Throughout the paper we use overbars to denote 
weak limits, in spaces that should be clear from the context.

\begin{lemma}\label{lem:prelim} 
Let $O$ be a bounded and open subset of $\R^M$ with $M\ge 1$.  
Suppose $g\colon \R\to (-\infty,\infty]$ is a lower semicontinuous 
convex function and $\Set{v_n}_{n\ge 1}$ is a sequence of 
functions on $O$ for which $v_n\weakto v$ in $L^1(O)$, $g(v_n)\in L^1(O)$ for each 
$n$, $g(v_n)\weakto \overline{g(v)}$ in $L^1(O)$. Then 
$g(v)\le \overline{g(v)}$ a.e.~on $O$, $g(v)\in L^1(O)$, and
$\int_O g(v)\ dy \le \liminf_{n\to\infty} \int_O g(v_n) \ dy$. 
If, in addition, $g$ is strictly convex on an open interval
$(a,b)\subset \R$ and $g(v)=\overline{g(v)}$ a.e.~on $O$, 
then, passing to a subsequence if necessary, 
$v_n(y)\to v(y)$ for a.e.~$y\in \Set{y\in O\mid v(y)\in (a,b)}$.
\end{lemma}

Let $X$ be a Banach space and $X^\star$ its dual.  The space
$X^\star$ equipped with the weak-$\star$ topology is denoted by
$X^\star_{\mathrm{weak}}$, while $X$ equipped with the weak topology
is denoted by $X_{\mathrm{weak}}$. By the Banach-Alaoglu theorem, 
bounded balls in $X^\star$ are $\sigma(X^\star,X)$-compact.  If $X$
separable, the weak-$\star$ topology is metrizable on bounded
sets in $X^\star$, which makes it possible to consider the metric space
$C\left([0,T];X^\star_{\mathrm{weak}}\right)$ of functions $v:[0,T]\to
X^\star$ that are continuous with respect to the weak topology. We
have $v_n\to v$ in $C\left([0,T];X^\star_{\mathrm{weak}}\right)$ if
$\langle v_n(t),\phi \rangle_{X^\star,X}\to \langle v(t),\phi
\rangle_{X^\star,X}$ uniformly with respect to $t$, for any $\phi\in
X$. The succeding lemma is a consequence of the Arzel\`a-Ascoli
theorem:

\begin{lemma}\label{lem:timecompactness}
Let $X$ be a separable Banach space, and suppose $v_n\colon [0,T]\to
X^\star$, $n=1,2,\dots$, is a sequence for which 
$\norm{v_n}_{L^\infty([0,T];X^\star)}\le C$, for some constant $C$ independent of $n$. 
Suppose the sequence $[0,T]\ni t\mapsto \langle v_n(t),\Phi \rangle_{X^\star,X}$, $n=1,2,\dots$, 
is equi-continuous for every $\Phi$ that belongs to a dense subset of $X$.  
Then $v_n$ belongs to $C\left([0,T];X^\star_{\mathrm{weak}}\right)$ for every
$n$, and there exists a function $v\in 
C\left([0,T];X^\star_{\mathrm{weak}}\right)$ such that along a 
subsequence as $n\to \infty$ there holds $v_n\to v$ in 
$C\left([0,T];X^\star_{\mathrm{weak}}\right)$.
\end{lemma}

Later we frequently obtain a priori estimates for a sequence $\Set{v_n}_{n\ge 1}$ 
that we make known as ``$v_n\inb X$'' for a given functional space $X$. What this really means is that 
we have a bound on $\norm{v_n}_X$ that is independent of $n$.

The following discrete version of a lemma 
due to Lions \cite[Lemma 5.1]{Lions:1998ga} 
will prove useful in the convergence analysis.

\begin{lemma}\label{lemma:aubinlions}
Given $T>0$ and a small number $h>0$, write 
$(0,T] = \cup_{k=1}^M(t_{k-1}, t_{k}]$ with $t_{k} = hk$ and $Mh = T$. 
Let $\{f_{h}\}_{h>0}^\infty$, $\{g_{h}\}_{h>0}^\infty $ be two sequences such that:
\begin{enumerate}
\item{} the mappings $t \mapsto g_{h}(t,x)$ and $t\mapsto f_{h}(t,x)$ 
are constant on each interval $(t_{k-1}, t_{k}],\ k=1, \ldots, M$.

\item{}$\{f_{h}\}_{h>0}$ and $\{g_{h}\}_{h>0}$ converge weakly to 
$f$ and $g$ in $L^{p_{1}}(0,T;L^{q_{1}}(\Om))$ and 
$L^{p_{2}}(0,T;L^{q_{2}}(\Om))$, respectively, 
where $1 < p_{1},q_{1}< \infty$ and
$$
\frac{1}{p_{1}} + \frac{1}{p_{2}} = \frac{1}{q_{1}} + \frac{1}{q_{2}} = 1.
$$

\item{} the discrete time derivative satisfies
$$
\frac{g_{h}(t,x) - g_{h}(t-h,x)}{h} \in_{b} L^1(0,T;W^{-1,1}(\Om))
$$

\item{}$\|f_{h}(t,x) - f_{h}(t,x-\xi)\|_{L^{p_{2}}(0,T;L^{q_{2}}(\Om))} 
\rightarrow 0$ as $|\xi|\rightarrow 0$, uniformly in $h$.
\end{enumerate}

Then $g_{h}f_{h} \weak gf$ in the sense of distributions on $\Dom$.
\end{lemma}
\begin{proof}
Let us introduce an auxiliary piecewise linear 
function $\widetilde{g}_{h}$ by setting
$$
\widetilde{g}_{h}(t,\cdot) = g_{h}(t_k)
+h^{-1}(t-t_k)\left(g_{h}(t_{k+1}) - g_{h}(t_k)\right), 
\quad t \in (t_k, t_{k+1}],
$$
for $k=0,\ldots,M-1$. Using property (3),
\begin{equation}\label{eq:eqeq}
	\begin{split}
		\left|\int_0^T\int_\Om (\widetilde{g}_h - g_h)\phi \ dxdt\right|
		& \leq	h\left|\int_{0}^T\int_{\Om} 
		\left(\frac{g_{h}(t,x) - g_{h}(t-h,x)}{h}\right) \phi \ dxdt\right| \\
		&\leq Ch \|\phi\|_{L^\infty(0,T;W^{1,\infty}(\Om))}, \qquad
		\phi \in C_0^\infty(\Om). 
	\end{split}
\end{equation}
Thus, $(\widetilde{g}_h - g_h) \weak 0$ as in the sense 
of distributions on $\Dom$ as $h \rightarrow 0$. 

Next, we write
\begin{equation*}
	g_{h}f_{h} 
	=\widetilde{g}_{h}f_{h} 
	+ (g_{h}- \widetilde{g}_{h})f_{h}.
\end{equation*}
By requirement (3), $\partial_{t}\widetilde{g}_{h} \in_{b} L^1(0,T;W^{-1,1}(\Om))$. 
This and requirement (4) allow us to apply a lemma due 
to Lions \cite[Lemma 5.1]{Lions:1998ga}, yielding
$$
f_{h}\widetilde{g}_{h} \weak fg, 
$$
in the sense of distributions on $\Dom$ as $h \rightarrow 0$. 

It only remains to prove that $(g_{h}- \widetilde{g}_{h})f_{h} \weak 0$ 
in the sense of distributions. For this purpose, set 
$f_{h}^\epsilon = f_{h}\star \kappa_{\epsilon}$, where
$\kappa_{\epsilon}$ is a standard smoothing 
kernel and $\star$ denotes the convolution product. 
We write
$$
(g_{h}-\widetilde{g}_{h})f_{h} = 
(g_{h}-\widetilde{g}_{h})f_{h}^\epsilon 
+ (g_{h}-\widetilde{g}_{h})(f_{h} - f_{h}^\epsilon).
$$
Now, requirement (4) yields
$$
\|f_{h}-f_{h}^\epsilon\|_{L^{p_{2}}(0,T;L^{q_{2}}(\Om))} \rightarrow 0 
\quad \text{as $\epsilon \rightarrow 0$},
$$
uniformly in $h$, and hence 
\begin{equation*}
	\begin{split}
			&\lim_{\epsilon \rightarrow 0}\lim_{h \rightarrow 0} 
			\int_0^T\int_\Om (g_{h}-\widetilde{g}_{h})(f_{h}-f_{h}^\epsilon)\phi\ dxdt =0.
	\end{split}
\end{equation*}
Thus, the proof is complete provided that 
\begin{equation*}
	\lim_{\epsilon \rightarrow 0}
	\lim_{h \rightarrow 0}	\int_0^T\int_\Om
	(g_{h}-\widetilde{g}_{h})f_{h}^\epsilon \phi \ dxdt = 0.
\end{equation*}
By a calculation similar to \eqref{eq:eqeq} 
we see that
$$
\left|\int_{0}^T\int_{\Om} (g_{h}
-\widetilde{g}_{h})f_{h}^\epsilon\ dxdt\right| 
\leq h^\frac{p_{2} - 1}{p_{2}} C 
\|f_{h}^\epsilon\|_{L^{p_{2}}(0,T;W^{1,\infty}(\Om))},
$$
where we have also applied Lemma \ref{lemma:inverse} (below)
to the time variable.  From this we can conclude 
that $(g_{h}- \widetilde{g}_{h})f_{h}^\epsilon \weak 0$ 
in the sense of distributions as $h \rightarrow 0$. This brings 
the proof to an end.
\end{proof}

\subsection{Weak and renormalized solutions}

\begin{definition}[Weak solutions]\label{def:weak}
A pair of functions $(\varrho,\vc{u})$ constitutes a weak solution
of the semi-stationary compressible Stokes 
system \eqref{eq:contequation}--\eqref{eq:momentumeq} 
with initial data \eqref{initial} provided that:
\begin{enumerate}
	\item $(\varrho,\vc{u}) \in L^\infty(0,T;L^\gamma(\Omega))\times L^2(0,T;\vc{W}^{1,2}_0(\Om)),$
	\item $\partial_t\varrho + \Div (\varrho\vc{u}) = 0$ in the weak 
	sense, i.e, $\forall \phi \in C^\infty([0,T)\times\overline{\Omega})$,
	\begin{equation}\label{eq:weak-rho}
		\int_{0}^T\int_{\Omega}\varrho \left( \phi_{t} + \vc{u}D\phi\right)\ dxdt
		+ \int_{\Omega}\varrho_{0}\phi|_{t=0}\ dx = 0;
	\end{equation}
		
	\item $-\mu \Delta \vc{u} - \lambda D\Div\vc{u} + Dp(\varrho) = \vc{f}$ in 
	the weak sense, i.e, $\forall \vc{\phi} \in \vc{C}_0^\infty([0,T)\times\Omega)$,

	\begin{equation}\label{eq:weak-u}
		\int_{0}^T\int_{\Omega}\mu\Grad \vc{u} \Grad \vc{\phi} 
		+ \left[(\mu + \lambda\Div \vc{u}-p (\varrho)\right]\Div \vc{\phi}\ dxdt 
		= \int_{0}^T\int_{\Omega}\vc{f}\vc{\phi}\ dxdt.
	\end{equation}
\end{enumerate}
\end{definition}

For the convergence analysis we shall also need the DiPerna-Lions concept of 
renormalized solutions of the continuity equation.  

\begin{definition}[Renormalized solutions]
\label{renormlizeddef}
Given $\vc{u}\in L^2(0,T;\vc{W}^{1,2}_{0}(\Omega))$, we 
say that $\varrho\in  L^\infty(0,T;L^\gamma(\Omega))$ 
is a renormalized solution of \eqref{eq:contequation} provided 
$$
B(\vrho)_t + \Div \left(B(\vrho)\vc{u}\right) = b(\vrho)\Div \vc{u}
\quad \text{in the sense of distributions on $[0,T)\times \overline{\Om}$},
$$
for any $B\in C[0,\infty)\cap C^1(0,\infty)$ with $B(0)=0$ and $b(\vrho) := B'(\vrho)\vrho - B(\vrho)$.
\end{definition}

We shall need the following well-known lemma \cite{Lions:1998ga} stating that 
square-integrable weak solutions $\vrho$ are also renormalized solutions.

\begin{lemma}
\label{lemma:feireisl}
Suppose $(\varrho,\vc{u})$ is a weak solution according to Definition \ref{def:weak}.
If $\varrho \in L^2((0,T)\times \Omega))$, then $\vrho$ is a renormalized solution 
according to Definition \ref{renormlizeddef}.
\end{lemma}

\begin{remark}
Regarding the continuity equation and the definitions of weak and renormalized solutions, 
we are requiring the equation to hold up to the boundary.  
\end{remark}

\subsection{On the equation {\rm div} v $=f$ }
Solutions to the following problem are vital 
to the upcoming convergence analysis:
\begin{equation}\label{eq:divf}
	\Div \vc{v} = f \quad \text{in $\Om$}, 
	\qquad \vc{v} = 0 \quad \text{on $\partial \Om$.}
\end{equation}
If $f \in L^p(\Om)$ with $\int_\Om f~dx = 0$, then a solution  to \eqref{eq:divf} can be constructed
through the Hodge decomposition
\begin{equation*}
	\vc{v} = \Grad s + \Curl \xi,
\end{equation*}
where $s \in \vc{H}^2(\Om)$ solves the Neumann Laplace problem, i.e.,
$$
\Delta s = f \quad \text{in $\Om$}, \qquad
\Grad s \cdot \nu = 0 \quad \text{on $\partial \Om$,}
$$ 
and $\xi \in \vc{H}^2(\Om)$ is determined such that 
$\vc{v}|_{\partial \Om} = 0$ (cf. \cite{Arnold}). 
Such a solution can be constructed using the Bogovskii 
solution operator \cite{Feireisl:2004oe}. Here, we define 
the solution operator $\Bop{\cdot}: L^p_0(\Om) \rightarrow \vc{W}^{1,p}_0(\Om)$ 
as one of the solutions to the problem
\begin{equation}\label{def:bop}
	\Div \Bop{\phi} = \phi\quad \text{in $\Om$},\qquad 
	\Bop{\phi} = 0 \quad \text{on $\partial \Om$.}
\end{equation}

We shall need solutions $\vc{v}$ satisfying $\Curl \vc{v} = 0$. 
Clearly, this is not compatible with the Dirichlet 
boundary condition. However, locally $\Curl$ free solutions 
can be constructed using the operator
$\Aop{\cdot}:L^p(\Om) \rightarrow \vc{W}^{1,p}(\Om)$,
\begin{equation}\label{def:aop}
	\Aop{\phi} = \Grad \Delta^{-1}\left[\phi\right],
\end{equation}
where $\Delta^{-1}$ is the inverse Neumann Laplace operator.

\subsection{Finite element spaces and some basic properties} \label{sec:CRM}

Let $E_h$ denote a shape regular tetrahedral mesh of $\Om$. Let 
$\Gamma_h^I=\Set{\Gamma \in \Gamma_h: \Gamma \not \subset \pOm}$ denote 
the set of internal faces in $E_h$. We will approximate the density 
in the space of piecewise constants on $E_{h}$ and 
denote this space by $Q_{h}(\Om)$. For the approximation of the velocity we 
use the Crouzeix--Raviart element space \cite{Crouzeix:1973qy}:
\begin{equation}\label{def:CR}
	\vc{V}_{h}(\Om) = \Set{\vc{v}_{h}\in L^2(\Om): 
	\vc{v}_{h}|_{E} \in \mathbb{P}_{1}^N(E), \ 
	\forall E \in E_{h},\ 
	\int_{\Gamma}\jump{\vc{v}_{h}}_\Gamma\ dS(x) = 0, \ 
	\forall \Gamma \in \Gamma^I_{h}},
\end{equation}
where $\jump{\cdot}_\Gamma$ denotes the jump across a face $\Gamma$.
To incorporate the boundary condition, we let the 
degrees of freedom of $\vc{V}_{h}(\Om)$ vanish at the boundary.
Consequently, the finite element method is nonconforming 
in the sense that the velocity approximation 
space is not a subspace of the corresponding 
continuous space, $\vc{W}_0^{1,2}(\Om)$.

We introduce the canonical interpolation operators
$$
\Pi_h^V: \vc{W}^{1,2}_0(\Om) \rightarrow \vc{V}_h(\Om), 
\qquad \Pi_h^Q: L^2(\Om) \rightarrow Q_h(\Om),
$$
defined by
\begin{equation}\label{eq:opdef}
	\begin{split}
		\int_\Gamma \Pi_h^V \vc{v}_h \ dS(x) & 
		= \int_\Gamma \vc{v}_h \ dS(x), 
		\quad\forall \Gamma \in \Gamma_h, \\
		\int_E \Pi_h^Q \phi~dx &= \int_E \phi \ dx, 
		\quad \forall E \in E_h.
	\end{split}
\end{equation}
Then, by virtue of \eqref{eq:opdef} and Stokes' theorem,
\begin{equation}\label{eq:commute}
	\Div_{h} \Pi_{h}^V\vc{v} = \Pi_{h}^Q \Div \vc{v},
	\qquad  
	\Curl_{h} \Pi_{h}^V\vc{v} = \Pi_{h}^Q \Curl \vc{v}, 
\end{equation}
for all $\vc{v} \in \vc{W}_0^{1,2}(\Om)$. Here, $\Curl_h$ and $\Div_h$ denote 
the curl and divergence operators, respectively, taken inside each element. 

Now, \eqref{eq:commute} immediately gives
\begin{equation*}
	\Div_h \Pi_h^V \Bop{q_h} = q_h, \qquad 
	\forall q_h \in Q_h(\Om)\cap \Set{\int_\Om q_h\ dx = 0},
\end{equation*}
and, away from the boundary,
\begin{equation*}
	\Div_h \Pi_h^V \Aop{q_h} = q_h, \qquad 
	\Curl_h \Pi_h^V \Aop{q_h} = 0, \qquad 
	\forall q_h \in Q_h(\Om),
\end{equation*}
where $\Bop{\cdot}$ and $\Aop{\cdot}$ are defined in \eqref{def:bop} and \eqref{def:aop}, respectively.
Consequently, this configuration of elements enables us to construct 
discrete analogs of the continuous operators \eqref{def:bop} and \eqref{def:aop}.

We associate to the space $\vc{V}_h(\Om)$ the following semi--norm and norm:
\begin{equation}\label{eq:norm}
	\begin{split}
		|\vc{v}_{h}|^2_{\vc{V}_{h}} & = 
		\|\Curl_{h} \vc{v}_{h}\|_{L^2(\Om)}^2 
		+\|\Div_{h} \vc{v}_{h}\|_{L^2(\Om)}^2  \\
		&\qquad + \sum_{\Gamma \in \Gamma_{h}}h^{\epsilon-1}
		\left(\|\sjump{\vc{v}_{h}\cdot \nu}\|_{L^2(\Gamma)}^2
		+\|\sjump{\vc{v}_{h}\times \nu}\|_{L^2(\Gamma)}^2\right),
		\\ \|\vc{v}_{h}\|_{\vc{V}_{h}}^2 & 
		=\|\vc{v}_{h}\|_{L^2(\Om)}^2 + |\vc{v}_{h}|^2_{\vc{V}_{h}}.
	\end{split}
\end{equation}

Let us now state some basic properties of the finite element spaces. 
We start by recalling from \cite{Brezzi:1991lr,Crouzeix:1973qy} 
a few interpolation error estimates. 	

\begin{lemma}\label{lemma:interpolationerror} 
There exists a constant $C>0$, depending only on the 
shape regularity of $E_h$ and $|\Om|$, such 
that for any $1\leq p<\infty$,
\begin{equation*}
	\begin{split}
		&\|\Pi_h^Q \phi-\phi \|_{L^p(\Om)} 
		\leq Ch\|\Grad \phi\|_{L^p(\Om)}, \\
		&\|\Pi_{h}^V \vc{v}-\vc{v}\|_{\vc{L}^p(\Om)} 
		+h\|\Grad_h (\Pi_{h}^V\vc{v}-\vc{v})\|_{\vc{L}^p(\Om)} 
		\leq ch^s\|\Grad^s\vc{v}\|_{\vc{L}^p(E)},\quad s = 1,2,
	\end{split}
\end{equation*}
for all $\phi \in W^{1,p}(\Om)$ and $\vc{v} \in \vc{W}^{s,p}(E)$.
Here, $\Grad_h$ is the gradient operator taken inside 
each element. 
\end{lemma}

By scaling arguments, the trace theorem, and 
the Poincar\'e inequality, we obtain

\begin{lemma}\label{lemma:toolbox}
For any $E \in E_{h}$ and $\phi \in W^{1,2}(E)$, 
we have the following inequalities:
\begin{enumerate}
\item{}trace inequality,
$$
\|\phi\|_{L^2(\Gamma)} \leq 
ch_{E}^{-\frac{1}{2}}\left(\|\phi\|_{L^2(E)}
+h_{E}\|\Grad \phi\|_{\vc{L}^2(E)}\right), 
\quad \forall 
\Gamma \in \Gamma_{h}\cap \partial E.
$$
\item{} Poincar\'e inequality,
$$
\left\|\phi - \frac{1}{|E|}\int_{E}\phi\ dx \right\|_{L^2(E)} 
\leq Ch_{E}\|\Grad \phi\|_{\vc{L}^2(E)}.
$$
\end{enumerate}
In both estimates, $h_E$ is the diameter of the element $E$.
\end{lemma}

\begin{lemma}\label{lemma:inverse}
There exists a positive constant $C$, depending only on the shape 
regularity of $E_h$, such that for $1\leq q,p \leq \infty$ 
and $r= 0,1$,
\begin{equation*}
	\norm{\phi_h}_{W^{r,p}(E)} 
	\leq C h^{-r + \min\{0, \frac{N}{p}-\frac{N}{q}\}}
	\norm{\phi_h}_{L^q(E)}, 
\end{equation*}
for any $E \in E_h$ and all polynomial 
functions $\phi_h \in \mathbb{P}_k(E)$, $k=0,1,\ldots$.
\end{lemma}

\begin{lemma}\label{lemma:consistency}
Let $\{\vc{v}_{h}\}_{h>0}$ be a sequence in $\vc{V}_{h}(\Om)$. Assume 
that there is a constant $C >0$, independent of $h$, such that 
$\|\vc{v}_{h}\|_{\vc{V}_{h}} \leq C$. Then there exists a 
function $\vc{v} \in \vc{W}^{1,2}_0(\Om)$ such that, by 
passing to a subsequence as $h\to0$ if necessary,
$$
\text{$\vc{v}_{h} \weak \vc{v}$ in $\vc{L}^2(\Om)$}, \quad
\text{$\Curl_{h} \vc{v}_{h} \weak \Curl \vc{v}$ in $\vc{L}^2(\Om)$}, \quad 
\text{$\Div_{h}\vc{v}_{h} \weak \Div \vc{v}$ in $L^2(\Om)$}.
$$
\end{lemma}

\begin{proof}
As $\|\vc{v}_{h}\|_{\vc{V}_{h}}$ is bounded independently of $h$, it follows that 
$\vc{v}_{h} \in _{b} \vc{L}^2(\Om)$, $\Curl_{h} \vc{v}_{h} \in_{b} \vc{L}^2(\Om)$, 
and $\Div_{h} \vc{v}_{h} \in_{b} L^2(\Om)$. 
Thus, we have the existence of functions $\vc{v} \in \vc{L}^2(\Om)$, 
$\xi \in \vc{L}^2(\Om)$, and $\zeta \in L^2(\Om)$ such that, by passing 
to a subsequence if necessary,
$$
\vc{v}_{h} \weak \vc{v}, \quad \Curl_{h}\vc{v}_{h} \weak \xi, 
\quad \Div_{h} \vc{v}_{h} \weak \zeta.
$$
Once we make the identifications $\xi = \Curl \vc{v}$ and 
$\zeta = \Div \vc{v}$, the proof is complete.

Fix any $\phi \in W^{1,2}_{0}(\Om)$. 
An application of Green's theorem yields
\begin{align*}
	\int_{\Omega}\Curl_{h}\vc{v}_{h} \phi\ dx 
	& = \sum_{E \in E_{h}}\int_{E}\vc{v}_{h}\Curl \phi \ dx 
	+ \int_{\partial E}\phi(\vc{v}_{h}\times \nu)\ dS(x) \\
	& = \int_{\Omega}\vc{v}_{h}\Curl \phi\ dx 
	+ \sum_{\Gamma \in \Gamma^I_{h}}\int_{\Gamma}
	\phi \sjump{\vc{v}_{h} \times \nu}\ dS(x).
\end{align*}
By sending $h\rightarrow 0$ in the above identity, we discover
\begin{equation}\label{eq:curlconsistency}
	\int_{\Omega}\xi \phi \ dx = \int_{\Omega}\vc{v}\Curl \phi \ dx 
	+\lim_{h \rightarrow 0}\sum_{\Gamma \in \Gamma^I_{h}}
	\int_{\Gamma}\phi \sjump{\vc{v}_{h} \times \nu}\ dS(x).
\end{equation}

Utilizing the bound
\begin{equation}\label{eq:vazelin}
	h^{\epsilon-1}\sum_{\Gamma \in \Gamma^I_{h}}\int_{\Gamma}
	\sjump{\vc{v}_{h} \times \nu}^2
	\ dS(x)\leq C,
\end{equation}
cf.~\eqref{eq:norm}, and the second condition 
in \eqref{def:CR}, we control the last 
term of \eqref{eq:curlconsistency}:
\begin{align*}
	&\left|\sum_{\Gamma \in \Gamma^I_{h}}\int_{\Gamma}\phi 
	\sjump{\vc{v}_{h} \times \nu}\ dS(x)\right| 
	\\ & \quad 
	= \left|\sum_{\Gamma \in \Gamma^I_{h}}\int_{\Gamma}
	(\phi - \phi_{\Gamma}) \sjump{\vc{v}_{h} \times \nu}\ dS(x)\right| 
	\\ & \quad  
	\leq h^{-\frac{\epsilon}{2}}
	\left(\sum_{\Gamma \in \Gamma^I_{h}}h^{\epsilon-1}
	\int_{\Gamma}\sjump{\vc{v}_{h} \times \nu}^2  \right)^\frac{1}{2}
	\left(\sum_{\Gamma \in \Gamma^I_{h}}h\int_{\Gamma}
	|\phi - \phi_{\Gamma}|^2\ dS(x) \right)^\frac{1}{2},
\end{align*}
where $\{\phi_{\Gamma}\}_{\Gamma \in \Gamma_{h}}$ is a given set of real numbers. 
For each $\Gamma \in \Gamma_{h}$, let us take 
$\phi_\Gamma := \frac{1}{|E|}\int_{E}\phi \ dx$, where $E$ is arbitrarily 
fixed as one of the two elements sharing the edge $\Gamma$. 
Now, using Lemma \ref{lemma:toolbox} 
and \eqref{eq:vazelin}, we deduce
$$
\left|\sum_{\Gamma \in \Gamma^I_{h}}
\int_{\Gamma}\phi \sjump{\vc{v}_{h} \times \nu}\ dS(x)\right| 
\leq Ch^{-\frac{\epsilon}{2}}h\|\Grad \phi\|_{\vc{L}^2(\Omega)}.
$$
Hence,
$$
\lim_{h \rightarrow }\left|\sum_{\Gamma \in \Gamma^I_{h}}
\int_{\Gamma}\phi \sjump{\vc{v}_{h} \times \nu}\ dS(x)\right| =0.
$$
By \eqref{eq:curlconsistency}, this shows that
$$
\int_{\Omega}\xi \phi \ dx = \int_{\Omega}\vc{v}\Curl \phi \ dx,
$$
and so our claim follows, i.e., $\xi = \Curl \vc{v}$. 

By almost identical arguments we find that $\zeta = \Div \vc{v}$.
\end{proof}

The following lemma provides us with an 
estimate of the blow-up rate of $\Grad_{h}\vc{v}_{h}$, 
for any element $\vc{v}_h\in \vc{V}_h(\Om)$.

\begin{lemma}\label{lem:blowup}           
There exists a positive constant $C$, depending only on the shape 
regularity of $E_{h}$ and the size of $\Omega$, such that
\begin{equation*}
	\|\Grad_{h}\vc{v}_{h}\|_{\vc{L}^2(\Om)}^2 
	\leq Ch^{-1-\frac{\epsilon}{2}}\|\vc{v}_{h}\|_{\vc{L}^2(\Om)}
	\left(\sum_{\Gamma \in \Gamma^I_{h}}h^{\epsilon-1}
	\norm{\sjump{\vc{v}_{h}}}_{\vc{L}^2(\Gamma)}^2 \right)^\frac{1}{2},
\end{equation*}
for all $\vc{v}_h \in \vc{V}_h(\Om)$.
\end{lemma}

\begin{proof}
By the linearity of $\vc{v}_h|_E$, $\Delta \vc{v}_{h}|_E = 0$ 
$\forall E \in E_{h}$. Using this we can apply Green's 
theorem to deduce the bound
\begin{align*}
	\|\Grad_{h}\vc{v}_{h}\|_{\vc{L}^2(\Om)}^2
	&=\sum_{E\in E_h}\int_E \Grad_{h}\vc{v}_{h}\cdot\Grad_{h}\vc{v}_{h}\ dx
	= \sum_{E \in E_{h}}\int_{\partial E}(\Grad\vc{v}_{h} \cdot \nu)\vc{v}_{h}\ dS(x) 
	\\ & 
	= \sum_{\Gamma \in \Gamma^I_{h}}\int_{\Gamma}
	\left(\Grad \sjump{\vc{v}_{h}}\cdot \nu\right) \vc{v}_{h} \ dS(x) 
	\\ & 
	\leq \sum_{\Gamma \in \Gamma^I_{h}}\int_{\Gamma}\left|\Grad\jump{\vc{v}_{h}}_{\Gamma} 
	\cdot \nu\right||\vc{v}_{h}| \ dS(x) =: I.
\end{align*}
To obtain the third equality we have used that the average of $\vc{v}_h$ is continuous 
across internal faces. 
Since $\vc{v}_h \in \vc{V}_h(\Om)$, we know that $\Grad \jump{\vc{v}_h}_\Gamma$ is 
constant for all internal faces $\Gamma \in \Gamma_h$. Moreover, 
there must exist a point $b_\Gamma \in \Gamma$, for every $\Gamma \in \Gamma_h^I$,
such that $\jump{\vc{v}_h(b_\Gamma)}_\Gamma = 0$. 
By this and the Cauchy-Schwartz inequality, we deduce
\begin{equation*}
	\begin{split}
	I &\leq 
		 C\sum_{\Gamma \in \Gamma^I_{h}}\frac{1}{h}
		\|\vc{v}_{h}\|_{\vc{L}^2(\Gamma)}
		\|\sjump{\vc{v}_{h}}\|_{\vc{L}^2(\Gamma)}  \\
		&\leq Ch^{-1-\frac{\epsilon}{2}}
		\|\vc{v}_{h}\|_{\vc{L}^2(\Om)}
		\left(\sum_{\Gamma \in \Gamma^I_{h}}h^{\epsilon-1}
		\|\sjump{\vc{v}_{h}}\|_{\vc{L}^2(\Gamma)}^2 \right)^\frac{1}{2}.	
\end{split}
\end{equation*}
The last inequality is achieved thanks to the trace inequality
(1) in Lemma \ref{lemma:toolbox}, together with Lemma \ref{lemma:inverse}.
\end{proof}

Using the previous lemma, we can now establish a Poincar\'e inequality and 
a spatial compactness estimate.

\begin{lemma}\label{lemma:compactembedding}
There exists a positive constant $C$, depending only on 
the shape regularity of $E_{h}$ and the size of $\Omega$, 
such that for any $\xi \in \mathbb{R}^2$
\begin{equation}\label{eq:compact}
	\|\vc{v}_{h}(\cdot) - \vc{v}_{h}(\cdot-\xi)\|_{\vc{L}^2(\Om_\xi)} 
	\leq C|\xi|^{\frac{1}{2}-\frac{\epsilon}{4}}|\vc{v}_{h}|_{\vc{V}_{h}(\Om)}, 
	\quad \forall \vc{v}_{h} \in \vc{V}_{h}(\Om),
\end{equation}
where $\Om_\xi = \Set{x \in \Om: \dist(x, \partial \Om)>\xi}$. Moreover,
\begin{equation}\label{eq:poincare}
	\|\vc{v}_{h}\|_{\vc{L}^2(\Om)} \leq 
	C|\vc{v}_{h}|_{\vc{V}_{h}(\Om)}, 
	\quad \forall \vc{v}_{h} \in \vc{V}_{h}(\Om).
\end{equation}
\end{lemma}

\begin{proof}
Fix $h>0$, and select an arbitrary function $\vc{v}_h$ in $\vc{V}_h(\Om)$. 
For any $\xi$ we construct a new mesh $G_{h}$ such that 
each $G \in G_{h}$ is a subset of one and only one
element $E \in E_{h}$ (e.g., we can divide each element of $E \in E_{h}$ 
into a number of smaller elements). Moreover, we construct this 
new mesh $G_{h}$ such that 
\begin{equation}\label{eq:hxiassertion}
C^{-1}|\xi| \leq h_{G} \leq C|\xi|, \qquad \forall G \in G_{h},
\end{equation}
where $h_{G}$ denotes the diameter of the new 
element $G$ and the constant $C$ depends only 
on the shape--regularity of $E_{h}$.

Now, let $\vc{V}_{|\xi|}(\Om)$ denote the Crouzeix--Raviart element space on $G_{h}$ 
and denote by $\Pi_{|\xi|}^V: \vc{V}_{h}(\Om) \rightarrow \vc{V}_{|\xi|}(\Om)$
the canonical interpolation operator associated with $\vc{V}_{|\xi|}(\Om)$.

Denote by $h_{|\xi|}$ the maximal element diameter in $G_{h}$. 
From standard properties of the Crouzeix--Raviart element \cite{Stummel:1980fk}, we have
\begin{align*}
	\|\Pi^V_{|\xi|}\vc{v}_{h}(x)-\Pi^V_{|\xi|}\vc{v}_{h}(x-\xi)\|_{\vc{L}^2(\Om_\xi)}^2 
	& \leq (h_{|\xi|}^2+|\xi|^2)\sum_{G \in G_{h}}
	\|\Grad\Pi^V_{|\xi|}\vc{v}_{h}\|_{\vc{L}^2(G)}^2
	\\ & \leq (h_{|\xi|}^2+|\xi|^2)\|\Grad_{h}\vc{v}_{h}\|_{\vc{L}^2(\Om)}^2,
\end{align*}
where the second inequality follows from the 
properties of the operator $\Pi_{|\xi|}^V$.

Lemma \ref{lem:blowup} and the bounds \eqref{eq:hxiassertion} 
allow us to conclude the following estimate:
\begin{align*}
	&\|\Pi^V_{|\xi|}\vc{v}_{h}(x)-\Pi^V_{|\xi|}\vc{v}_{h}(x-\xi)\|_{\vc{L}^2(\Om_\xi)}^2 
	\\ &\quad \leq h_{|\xi|}^{-1 -\frac{\epsilon}{2}}
	(h_{|\xi|}^2 + |\xi|^2)\|\vc{v}_{h}\|_{\vc{L}^2(\Om)}
	\left(\sum_{\Gamma \in \Gamma^I_{h}}h^{\epsilon-1}
	\|\sjump{\vc{v}_{h}}\|_{\vc{L}^2(\Gamma)}^2 \right)^\frac{1}{2}\\
	&\quad \leq C|\xi|^{1-\frac{\epsilon}{2}}\|\vc{v}_{h}\|_{\vc{L}^2(\Om)}
	\left(\sum_{\Gamma \in \Gamma^I_{h}}h^{\epsilon-1}
	\|\sjump{\vc{v}_{h}}\|_{\vc{L}^2(\Gamma)}^2 \right)^\frac{1}{2}.
\end{align*}

Keeping in mind that $\vc{v}_{h}|_E \in \vc{W}^{1,2}(E)$, $\forall E \in E_{h}$, 
we can apply Lemma \ref{lemma:interpolationerror} and 
the previous estimate to obtain
\begin{equation}\label{eq:betweenbarkandwood}
	\begin{split}
		&\|\vc{v}_{h}(\cdot)-\vc{v}_{h}(\cdot-\xi)\|_{\vc{L}^2(\Om_\xi)}^2 \\
		&\quad  
		\leq 2\|\vc{v}_{h}-\Pi^V_{|\xi|}\vc{v}_{h}\|_{\vc{L}^2(\Om_\xi)}^2
		+\|\Pi^V_{|\xi|}\vc{v}_{h}(x)-\Pi^V_{|\xi|}
		\vc{v}_{h}(x-\xi)\|_{\vc{L}^2(\Om_\xi)}^2 \\
		&\quad 
		\leq C|\xi|^{1-\frac{\epsilon}{2}}\|\vc{v}_{h}\|_{\vc{L}^2(\Om)}
		\left(\sum_{\Gamma \in \Gamma^I_{h}}h^{\epsilon-1}
		\|\sjump{\vc{v}_{h}}\|_{\vc{L}^2(\Gamma)}^2 \right)^\frac{1}{2}.
	\end{split}
\end{equation}

Next, denote by $\vc{v}_h^\text{ext}$ the extension of $\vc{v}_h$ 
by zero to all of $\mathbb{R}^N$. By the previous calculations, we 
conclude that $\vc{v}_h^\text{ext}$ satisfies \eqref{eq:betweenbarkandwood} 
with $\Om_\xi$ replaced by $\mathbb{R}^N$ (keep in mind that the 
jump terms are only summed over internal faces). 
Thus, we can fix $|\xi|$ large in \eqref{eq:betweenbarkandwood} 
to discover that
\begin{align*}
	& \|\vc{v}^\text{ext}_{h}\|_{\vc{L}^2(\Om)}^2 
	\leq C|\operatorname{diam}(\Om)|^{1-\frac{\epsilon}{2}}
	\|\vc{v}^\text{ext}_{h}\|_{\vc{L}^2(\Om)}\left(\sum_{\Gamma \in
	\Gamma^I_{h}}h^{\epsilon-1}\|\sjump{\vc{v}^\text{ext}_{h}}\|_{\vc{L}^2(\Gamma)}^2 
	\right)^\frac{1}{2},
\end{align*}
and hence
\begin{align*}
	\|\vc{v}_{h}\|_{\vc{L}^2(\Om)}^2 
	& \leq C|\operatorname{diam}(\Om)|^{1-\frac{\epsilon}{2}}\left(\sum_{\Gamma \in
	\Gamma^I_{h}}h^{\epsilon-1}\|\sjump{\vc{v}_{h}}\|_{\vc{L}^2(\Gamma)}^2 \right) 
	\leq C|\operatorname{diam}(\Om)|^{1-\frac{\epsilon}{2}}\abs{\vc{v}_h}_{\vc{V}_h}^2,
\end{align*}
which is \eqref{eq:poincare}.

Finally, setting \eqref{eq:poincare} into \eqref{eq:betweenbarkandwood} 
gives \eqref{eq:compact}.
\end{proof}

We end this section with	
	
\begin{lemma}\label{lemma:jumpcontrol}
There exists a constant $C>0$, which depends only on the shape regularity of
$E_h$ and the size of $\Om$, such that 
for any $\vc{v}_h \in \vc{V}_h(\Om)$
\begin{align*}
	&\left|\sum_{\Gamma \in \Gamma^I_{h}}h^{\epsilon-1}\int_{\Gamma}
	\sjump{\vc{v}_{h}\cdot \nu}\sjump{\Pi_{h}^V\vc{w}\cdot \nu}
	+\sjump{\vc{v}_{h}\times \nu}
	\sjump{\Pi_{h}^V\vc{w}\times \nu} \ dS(x)\right| \\
	& \quad \leq C h^{\frac{\epsilon}{2}}\|\vc{v}_{h}\|_{\vc{V}_{h}(\Om)}
	\|\Grad \vc{w}\|_{\vc{L}^2(\Omega)}, 
	\qquad \forall \vc{w} \in \vc{W}^{1,2}_{0}(\Om).
\end{align*}
\end{lemma}

\begin{proof}
Using the H\"older inequality,
\begin{align*}
	&\left|\sum_{\Gamma \in \Gamma^I_{h}}h^{\epsilon-1}
	\int_{\Gamma}\sjump{\vc{v}_{h}\cdot \nu}
	\sjump{\Pi_{h}^V\vc{w}\cdot \nu}\ dS(x)\right| 
	\\ & \quad 
	\leq h^\frac{\epsilon}{2}
	\left(\sum_{\Gamma \in \Gamma^I_{h}}h^{\epsilon-1}
	\int_{\Gamma}\sjump{\vc{v}_{h}\cdot \nu}^2 \ dS(x) \right)^\frac{1}{2}
	\left(\sum_{\Gamma \in \Gamma^I_{h}}
	h^{-1}\int_{\Gamma}\sjump{\Pi_{h}^V\vc{w} }^2\ dS(x)\right)^\frac{1}{2} \\
	&\quad
	\leq h^\frac{\epsilon}{2}
	\left(\sum_{\Gamma \in \Gamma^I_{h}}h^{\epsilon-1}
	\int_{\Gamma}\sjump{\vc{v}_{h}\cdot \nu}^2 \ dS(x) \right)^\frac{1}{2}\\
	&\quad \qquad\qquad  
	\times\left(\sum_{E \in E_{h}}
	h^{-1}\int_{\partial E}\abs{\Pi_{h}^V\vc{w}
	-\vc{w}}^2\ dS(x)\right)^\frac{1}{2}.
\end{align*}
To obtain the last inequality, we have applied the calculation
\begin{equation*}
	\begin{split}
		\jump{\Pi_h^V\vc{w}}_\Gamma^2 &= \left|(\Pi_h^V\vc{w})|_{\partial E^+} - \vc{w} + \vc{w}- (\Pi_h^V\vc{w})|_{\partial E^-}\right|^2 \\
		&\leq |(\Pi_h^V\vc{w})|_{\partial E^+} - \vc{w}|^2 + |(\Pi_h^V\vc{w})|_{\partial E^-} - \vc{w}|^2,
	\end{split}
\end{equation*}
where $E^+$ and $E^-$ are the two element sharing the face $\Gamma$.

By using (1) in Lemma \ref{lemma:toolbox}, we further deduce that
\begin{align*}
	&\left|\sum_{\Gamma \in \Gamma_{h}}h^{\epsilon-1}\int_{\Gamma}
	\sjump{\vc{v}_{h}\cdot \nu}\sjump{\Pi_{h}^V\vc{w}\cdot \nu}\ dS(x)\right| \\
	&\leq h^\frac{\epsilon}{2}\left(\sum_{\Gamma \in \Gamma^I_{h}}h^{\epsilon-1}
	\int_{\Gamma}\sjump{\vc{v}_{h}\cdot \nu}^2 \ dS(x) \right)^\frac{1}{2} \\
	&\qquad \times \left(\sum_{E \in E_{h}}h^{-2}\|\Pi_{h}^V \vc{w} - \vc{w}\|_{\vc{L}^2(E)}^2 
	+ \|\Grad(\Pi_{h}^V \vc{w} - \vc{w})\|_{\vc{L}^2(E)}^2\right)^\frac{1}{2} \\
	& \leq h^\frac{\epsilon}{2}C\|\vc{v}_{h}\|_{\vc{V}_{h}}
	\|\Grad \vc{w}\|_{L^2(\Omega)},
\end{align*}
where the last inequality follows from Lemma \ref{lemma:interpolationerror}.

By analogous calculations for the tangential jumps,
$$
\left|\sum_{\Gamma \in \Gamma^I_{h}}h^{\epsilon-1}\int_{\Gamma}
\sjump{\vc{v}_{h}\times \nu}\sjump{\Pi_{h}^V\vc{w}\times \nu}\ dS(x)\right| 
\leq h^\frac{\epsilon}{2}C\|\vc{v}_{h}\|_{\vc{V}_{h}}
\|\Grad \vc{w}\|_{L^2(\Omega)}.
$$
This concludes the proof.
\end{proof}


\section{Numerical method and main result}\label{sec:numerical-method}


Given a time step $\Dt>0$, we discretize the time interval $[0,T]$ in 
terms of the points $t^m=m\Dt$, $m=0,\dots,M$, where it is assumed that $M\Dt=T$.  
Regarding the spatial discretization, we let $\{E_{h}\}_{h}$ be a shape regular 
family of tetrahedral meshes of $\Omega$,  where $h$ is the maximal diameter.
It will be a standing assumption that $h$ and $\Delta t$ are 
related like $\Delta t = c h$ for some constant $c$. 
By shape regular we mean the existence of a constant 
$\kappa > 0$ such that every $E \in E_{h}$ contains a ball of radius 
$\lambda_{E} \geq \frac{h_{E}}{\kappa}$, where $h_{E}$ is the diameter of $E$. 
Furthermore, we let $\Gamma_{h}$ denote the set of faces in $E_{h}$. 
Throughout the paper, we will use ``three dimensional" terminology 
(tetrahedron, face, etc.) when referring to both the three dimensional 
case and the two dimensional case (triangle, edge, etc). 

On each element $E \in E_{h}$, we denote by $Q(E)$ the constants 
on $E$. The functions that are piecewise constant with 
respect to the elements of a mesh $E_{h}$ are 
denoted by $Q_h(\Om)$. We denote by $\vc{V}_h(\Om)$ 
the Crouzeix--Raviart finite element space \eqref{def:CR} formed on $E_h$.
To incorporate the boundary condition, we let the 
degrees of freedom of $\vc{V}_h(\Om$ vanish at the boundary: 
$$
\int_\Gamma \vc{v}_h~dS(x) = 0, 
\quad \forall \Gamma \in \Gamma_h\cap \partial \Om, 
\quad \forall \vc{v}_h \in \vc{V}_h(\Om).
$$

We shall need to introduce some additional notation 
related to the discontinuous Galerkin method. 
Concerning the boundary $\partial E$ of an element $E$, we write $f_{+}$ 
for the trace of the function $f$ achieved from within the element $E$ 
and $f_{-}$ for the trace of $f$ achieved from outside $E$. 
Concerning a face $\Gamma$ that is shared between two 
elements $E_{-}$ and $E_{+}$, we will write $f_{+}$ for 
the trace of $f$ achieved from within $E_{+}$ and $f_{-}$ for the trace
of $f$ achieved from within $E_{-}$. Here $E_{-}$ and $E_{+}$ are 
defined such that $\nu$ points from $E_{-}$ to $E_{+}$, where $\nu$ is 
fixed (throughout) as one of the two possible 
normal components on each face $\Gamma$.
We also write $\jump{f}_{\Gamma}= f_{+} - f_{-}$ for the jump of $f$ 
across the face $\Gamma$, while forward time-differencing 
of $f$ is denoted by $\jump{f^m} = f^{m+1} - f^m$ and
$\Pth f^m = \frac{\jump{f^m}}{\Dt}$. The set of inner faces of $\Gamma_h$ will be denoted by 
$\Gamma_h^I = \Set{\Gamma \in \Gamma_h; \Gamma \not \subset \partial \Om}$.


\begin{definition}[Numerical scheme]\label{def:num-scheme}
Let $\Set{\vrho^0_h(x)}_{h>0}$ be a sequence (of piecewise constant 
functions) in $Q_{h}(\Omega)$ that satisfies $\vrho_h^0>0$ for each fixed $h>0$ 
and $\varrho^0_h\to \vrho_0$ a.e.~in $\Om$ and in $L^1(\Om)$ 
as $h\to 0$. Set $\vc{f}_{h}(t,\cdot)=\vc{f}_{h}^m(\cdot) 
:= \frac{1}{\Delta t}\int_{t^{m-1}}^{t^m} \Pi_h^Q \vc{f}(s,\cdot) \ ds$, 
for $t\in (t_{m-1},t_m)$, $m=1,\ldots,M$.

Determine functions 
$$
(\varrho^m_{h},\vc{u}^m_{h}) \in Q_{h}(\Omega)
\times \vc{V}_{h}(\Omega), \quad m=1,\dots,M,
$$
such that for all $\phi_{h} \in Q_{h}(\Omega)$,
\begin{equation}\label{FEM:contequation}
	\begin{split}
		&\int_\Omega \Pth (\vrho_{h}^m) \phi_{h}\ dx 
		 =\Delta t\sum_{\Gamma \in \Gamma^I_h}\int_\Gamma 
		\left(\varrho^m_{-}(\vc{u}_h^m \cdot \nu)^+_h 
		+\varrho^m_+(\vc{u}^{m}_h \cdot \nu)_h^-\right)
		\jump{\phi_{h}}_\Gamma\ dS(x).
	\end{split}
\end{equation}
and for all $\vc{v}_{h} \in  \vc{V}_{h}(\Omega)$,
\begin{equation}\label{FEM:momentumeq}
	\begin{split}
		&\int_{\Om} \mu\Curl_{h} \vc{u}^m_{h}\Curl_{h} \vc{v}_{h} 
		+ \left[(\mu + \lambda)\Div_{h} \vc{u}^m_{h}
		-p(\varrho^m_{h})\right]\Div_{h} \vc{v}_{h}\ dx \\
		&\qquad \qquad 
		+\mu \sum_{\Gamma \in \Gamma^I_{h}} h^{\epsilon -1}
		\int_{\Gamma}\sjump{\vc{u}^m_{h}\cdot \nu}\sjump{\vc{v}_{h}\cdot \nu} 
		+ \sjump{\vc{u}^m_{h}\times \nu}\sjump{\vc{v}_{h}\times \nu}\ dS(x) \\
		& \quad = \int_{\Omega}\vc{f}^m_{h}\vc{v}_{h}\ dx, 
	\end{split}
\end{equation}
for $m=1,\dots,M$. 

In \eqref{FEM:contequation}, we have introduced the notation
\begin{equation}\label{eq:velocity-average} 
	(\vc{u}_{h} \cdot \nu)_h^{\pm}=\left(\frac{1}{|\Gamma|}
	\int_{\Gamma} \vc{u}_{h}\cdot \nu \ dS(x)\right)^{\pm},
\end{equation} 
where $a^{+}=\max(a,0)$ and $a^{-}=\min(a,0)$.
\end{definition}

\begin{remark}
Since the normal velocity components $(\vu_h^m\cdot \nu)$ 
are discontinuous across element faces, the 
continuity method \eqref{FEM:contequation} 
approximates $(\vr\vu \cdot \nu)$ using instead 
the average normal velocity $\frac{1}{|\Gamma|}
\int_\Gamma(\vu_h^m \cdot \nu) \ dS(x)$, cf.~\eqref{eq:velocity-average}, 
and traces of $\vr$ are taken in the upwind direction with 
respect to the average normal velocity. 
\end{remark}

We now make an observation that will simplify the subsequent analysis. 
Let $\mathcal{N}_h(\Om)$ denote the lowest order div conforming 
Nedelec finite element space of the first kind \cite{Nedelec:1980ec,Karlsen1} on $E_h$. 
In two dimensions, $\mathcal{N}_h(\Om)$ is the Raviart--Thomas space. 

We will need the interpolation operator 
$\Pi_h^\mathcal{N}:\vc{V}_h(\Om)
\rightarrow \mathcal{N}_h(\Om)$ defined by
$$
\int_\Gamma \left(\Pi_h^\mathcal{N}\vc{v}_h\right)\cdot \nu \ dS(x) 
= \int_\Gamma \vc{v}_h\cdot \nu \ dS(x), 
\quad \forall \Gamma \in \Gamma_h.
$$
Then, by definition, the interpolated velocity
\begin{equation}\label{eq:nedelec}
	\widetilde{\vu_h^m} 
	:= \Pi_h^\mathcal{N}\vu_h^m
\end{equation}
satisfies
\begin{equation}\label{eq:normalsagree}
	(\widetilde{\vu_h^m}\cdot \nu)^\pm 
	=\left(\widetilde{\vc{u}_h^m}\cdot \nu \right)_h^\pm 
	=(\vc{u}_h^m \cdot \nu)_h^\pm,
\end{equation}
where the first equality is valid since $\widetilde{\vu_h^m} \cdot \nu$ is 
constant on each $\Gamma \in \Gamma_h$. 
Since both element spaces have piecewise constant divergence, 
a direct calculation yields
\begin{equation}\label{eq:divagree}
	\Div \widetilde{\vu_h^m} 
	= \Div_h \vu_h^m.
\end{equation}

Now, setting \eqref{eq:normalsagree} into the 
continuity method \eqref{FEM:contequation} leads to the relation
\begin{equation}\label{FEM:oldcontequation}
	\begin{split}
		&\int_\Omega \Pth (\vrho_{h}^m) \phi_{h}\ dx 
		=\Delta t\sum_{\Gamma \in \Gamma^I_h}\int_\Gamma 
		\left(\varrho^m_{-}(\widetilde{\vc{u}_h^m} \cdot \nu)^+
		+\varrho^m_+(\widetilde{\vc{u}^{m}_h} \cdot \nu)^-\right)
		\jump{\phi_{h}}_\Gamma\ dS(x),
	\end{split}
\end{equation}
for all $\phi_h \in Q_h(\Om)$. Hence, we can think of 
the pair $(\vr_h^m, \widetilde{\vu_h^m})$ as a solution 
to a continuity method in which $\mathcal{N}_h(\Om)$ is 
used to approximate the velocity. In fact, \eqref{FEM:oldcontequation} is 
the method examined in \cite{Karlsen1}. 
We will frequently utilize \eqref{FEM:oldcontequation}, instead 
of \eqref{FEM:contequation}, to easily obtain properties 
of our continuity approximations.

For each fixed $h>0$, the numerical solution 
$\Set{(\varrho^m_{h},\vc{u}^m_{h})}_{m=0}^M$
is extended to the whole of $(0,T]\times \Omega$ by setting 
\begin{equation}\label{eq:num-scheme-II}
	(\varrho_{h},\vc{u}_{h})(t)=(\varrho^m_{h},\vc{u}^m_{h}), 
	\qquad t\in (t_{m-1},t_m], \quad m=1,\dots,M.
\end{equation}
In addition, we set $\varrho_{h}(0)= \varrho^0_{h}$. 

\subsection{Main result}
Our main result is that, passing if necessary to a subsequence, 
$\Set{(\varrho_{h},\vc{u}_{h})}_{h>0}$ 
converges to a weak solution. More precisely, we will prove

\begin{theorem}[Convergence]\label{theorem:mainconvergence}
Suppose $\vc{f}\in \vc{L}^2(\Dom)$, and $\vrho_0 \in L^\gamma (\Om)$ 
with $\gamma > 1$. Let $\Set{(\varrho_{h},\vc{u}_{h})}_{h>0}$ 
be a sequence of numerical solutions 
constructed according to \eqref{eq:num-scheme-II} 
and Definition \ref {def:num-scheme}. 
Then, passing if necessary to a subsequence as $h\to 0$, 
$\vc{u}_{h} \weak \vc{u}$ in $L^2(0,T;\vc{L}^{2}(\Omega))$, 
$\varrho_{h}\vc{u}_{h} \weak \varrho\vc{u}$ in the sense 
of distributions on $\Dom$, and $\varrho_{h} \rightarrow \varrho$
a.e.~in $\Dom$, where the limit pair $(\vrho, \vc{u})$  is a weak 
solution as stated in Definition \ref{def:weak}.
\end{theorem}

This theorem will be a consequence of the results 
proved in Sections \ref{sec:basic} and \ref{sec:conv}.

\subsection{The numerical method is well--defined}
We now turn to the existence of a solution to the discrete problem. 
However, we commence with the following easy lemma 
providing a positive lower bound for the density.
 
\begin{lemma} \label{lemma:vrho-props}
Fix any $m=1,\dots,M$ and suppose $\varrho^{m-1}_{h} 
\in Q_{h}(\Om)$, $\vc{u}^m_{h} \in \vc{V}_{h}(\Om)$ are given bounded functions. 
Then the solution $\varrho^{m}_{h} \in Q_{h}(\Om)$ of the discontinuous 
Galerkin method \eqref{FEM:contequation} satisfies
$$
\min_{x \in \Omega}\varrho_{h}^m 
\geq \min_{x \in \Omega}\varrho_{h}^{m-1}\left(\frac{1}{1 + 
\Delta t \|\Div_h \vc{u}^m_{h}\|_{L^\infty(\Omega)}}\right).
$$
Consequently, if $\varrho^{m-1}_{h}(\cdot)>0$, then $\varrho^{m}_{h}(\cdot)>0$.
\end{lemma}

\begin{proof}
Let $\widetilde{\vu_h^m}$ be given by \eqref{eq:nedelec}. 
Then, since $(\vr_h^m,\widetilde{\vu_h^m})$ satisfies \eqref{FEM:oldcontequation},
Lemma 4.1 in \cite{Karlsen1} can be applied. This concludes the proof.
\end{proof}

\begin{lemma}
For each fixed $h > 0$, there exists a solution 
$$
(\varrho^m_{h},\vc{u}^m_{h}) 
\in Q_{h}(\Omega)
\times \vc{V}_{h}(\Omega), \quad 
\vrho^m_h(\cdot)>0, \quad m=1,\dots,M,
$$
to the discrete problem posed in Definition \ref {def:num-scheme}. 
\end{lemma}

\begin{proof}
As in the proof of \cite[Lemma 4.2]{Karlsen1}, the existence of 
a solution is established using a topological degree argument. 
By the arguments corresponding to those in \cite[Lemma 4.2]{Karlsen1}, 
one reduces the problem to proving existence 
of a solution $\vc{u}_h \in \vc{V}_h(\Om)$ to the linear system:
\begin{equation}\label{eq:thelinsys}
	\begin{split}
		a(\vc{u}_h, \vc{v}_h)  
		& := \int_{\Om} \mu\Curl_{h} \vc{u}_{h}\Curl_{h} \vc{v}_{h} 
		+ (\mu + \lambda)\Div_{h} \vc{u}_{h}\Div_{h} \vc{v}_{h}\ dx \\
		&\qquad 
		+\mu \sum_{\Gamma \in \Gamma^I_{h}} 
		h^{\epsilon-1}\int_{\Gamma}\sjump{\vc{u}_{h}\cdot \nu}
		\sjump{\vc{v}_{h}\cdot \nu} 
		+ \sjump{\vc{u}_{h}\times \nu}\sjump{\vc{v}_{h}\times \nu}\ dS(x) \\
		&\qquad\qquad 
		= \int_{\Omega}\vc{g}\vc{v}_{h}\ dx, 
		\quad \forall \vc{v}_h \in \vc{V}_h(\Om),
	\end{split}
\end{equation}
where $\vc{g} \in \vc{L}^2(\Om)$ is given. 

Now, the bilinear form $a(\cdot, \cdot)$
is clearly bounded on the space $\vc{V}_h(\Om) \times \vc{V}_h(\Om)$
equipped with the norm $\|\cdot \|_{\vc{V}_h}$. By an application of the
Poincar\'e inequality \eqref{eq:poincare}, we also have the existence of 
a constant $C$, independent of $h$, such that
\begin{equation*}
	a(\vu_h, \vu_h) \geq C \|\vu_h\|_{\vc{V}_h}^2.
\end{equation*}
Hence, the bilinear form $a(\cdot, \cdot)$ is coercive on $\vc{V}_h$, and 
the existence of a function $\vu_h \in \vc{V}_h(\Om)$ satisfying \eqref{eq:thelinsys} follows.

Since the remaining part of the topological degree 
argument is very similar to that found in \cite[Lemma 4.2]{Karlsen1}, we 
omit the details.
\end{proof}

\section{Basic estimates}\label{sec:basic}
In this section we establish various a priori estimates for the 
discrete problem given in Definition \ref{def:num-scheme}, including a 
basic energy estimate and a higher integrability 
estimate for the density approximations. 

We begin with a renormalized formulation of 
the continuity method \eqref{FEM:contequation}.
 
\begin{lemma}[Renormalized continuity method]
Fix any $m=1,\ldots,M$ and let $(\varrho_{h}^m,\vc{u}_{h}^{m}) \in Q_{h} 
\times \vc{V}_{h}$ satisfy the continuity method \eqref{FEM:contequation}. 
Then $(\varrho_{h}^m, \vc{u}_{h}^{m})$ also 
satisfies the renormalized formulation
\begin{equation}\label{FEM:renormalized}
	\begin{split}
		&\int_{\Omega}B(\vrho_{h}^m)\phi_{h}\ dx \\
		&\qquad - \Delta t\sum_{\Gamma \in \Gamma^I_{h}} 
		\int_\Gamma \left(B(\vrho^m_-)(\vc{u}^{m}_h \cdot \nu)_h^+ 
		+ B(\vrho^m_+)(\vc{u}^{m}_h \cdot \nu)_h^-\right) \jump{\phi_{h}}_\Gamma\ dx \\
		&\qquad + \Delta t \int_{\Omega}b(\vrho_{h}^m)\Div_h \vc{u}^{m}_{h}\phi_{h}\ dx 
		+  \int_{\Omega}B''(\xi(\vrho_{h}^m, \vrho_{h}^{m-1})) 
		\jump{\vrho_{h}^{m-1}}^2\phi_{h}\ dx  \\
		&\qquad  
		+\Delta t\sum_{\Gamma \in \Gamma^I_{h}}
		\int_{\Gamma}B''(\xi^\Gamma(\vrho^m_{+},\vrho^m_{-}))
		\jump{\vrho^m_{h}}^2_{\Gamma}(\phi_{h})_{-}(\vc{u}_{h}^{m} \cdot \nu)_h^+ \\
		&\qquad \qquad \qquad \qquad\qquad 
		-B''(\xi^\Gamma(\vrho^m_{-},
		\vrho^m_{+}))\jump{\vrho^m_{h}}^2_{\Gamma}(\phi_{h})_{+}
		(\vc{u}_{h}^{m} \cdot \nu)_h^-\ dS(x) \\
		& \qquad\qquad 
		= \int_{\Omega}B(\vrho_{h}^{m-1})\phi_{h}\ dx,
		\qquad \forall \phi_{h} \in Q_{h}(\Omega),
	\end{split}
\end{equation} 
for any $B\in C[0,\infty)\cap C^2(0,\infty)$ with $B(0)=0$ 
and $b(\vrho) :=\vrho B'(\vrho)-B(\vrho)$. 
Given two positive real numbers $a_1$ and $a_2$, we use 
$\xi(a_1,a_2)$ and $\xi^\Gamma(a_1,a_2)$ to denote two corresponding 
numbers between $a_1$ and $a_2$ that 
arise from second order Taylor expansions utilized in the proof.
\end{lemma}

\begin{proof}
Recall the definition of $\widetilde{\vu_h^m}$, cf.~\eqref{eq:nedelec}. 
By taking $B'(\vr_h^m)\phi_h$ as test function in \eqref{FEM:oldcontequation} and 
repeating the proof of Lemma 5.1 in \cite{Karlsen1}, we 
obtain \eqref{FEM:renormalized} with $\vu_h^m$ replaced by $\widetilde{\vu_h^m}$.  
In view of \eqref{eq:normalsagree} and \eqref{eq:divagree}, this 
is identical to \eqref{FEM:renormalized}.
\end{proof}

\begin{lemma}[Stability]\label{lemma:stability} \solutiontext 
For $\vrho> 0$, set $P(\vrho):=\frac{a}{\gamma-1}\vrho^\gamma$.
For any $m=1,\dots,M$, we have 
\begin{equation}\label{eq:stabilityincerrors}
	\begin{split}
		&\int_{\Omega} P(\varrho_{h}^m)\ dx 
		+\frac{\mu}{2}\sum_{k=1}^m \Delta t 
		\|\vc{u}_{h}^k\|_{\vc{V}_{h}(\Omega)}^2 
		+\mathcal{N}^m_{\text{diffusion}} \\
		& \qquad 
		\leq \int_{\Omega} P(\varrho_{0}) \ dx 
		+C\sum_{k=1}^m\Delta t\|\vc{f}_{h}^k\|_{\vc{L}^2(\Om)}^2,
	\end{split}
\end{equation}
where the numerical diffusion term 
$\mathcal{N}^m_{\text{diffusion}}\ge 0$ takes the form
\begin{align*}
	\mathcal{N}^m_{\text{diffusion}} 
	& =\sum_{k=1}^{m}\int_{\Omega} 
	P''(\xi(\varrho_{h}^k,\varrho_{h}^{k-1}))\jump{\varrho^{k-1}_{h}}^2\ dx 
	\\ & \qquad +\sum_{k=1}^{m}\sum_{\Gamma \in \Gamma^I_{h}}
	\Delta t\int_{\Gamma}P''(\varrho^k_{\dagger})\jump{\varrho^k_{h}}_{\Gamma}^2
	\left((\vc{u}^k_{h} \cdot \nu)_h^+ 
	-(\vc{u}^k_{h} \cdot \nu)_h^-\right) \ dS(x).
\end{align*}
In particular, $\varrho_{h}\inb L^\infty(0,T;L^\gamma(\Omega))$.
\end{lemma}

\begin{proof}
Since $P'(\rho)\rho - P(\rho)=p(\rho)$ and $\vrho_h>0$, taking 
$\phi_{h} \equiv 1$ in \eqref{FEM:renormalized} yields
\begin{equation}\label{eq:stabilityeq1}
	\begin{split}
		& \int_{\Omega}P(\varrho_{h}^k)\ dx
		+ \Delta t \int_{\Om}p(\varrho_{h}^k)\Div_{h} \vc{u}^{k}_{h}\ dx 
		+\int_{\Omega}P''(\xi(\varrho_{h}^k, \varrho_{h}^{k-1}))
		\jump{\varrho_{h}^{k-1}}^2 dx \\
		&\qquad +\Delta t\sum_{\Gamma \in \Gamma^I_{h}}
		\int_{\Gamma}P''(\xi^\Gamma (\varrho^m_{+}, \varrho^m_{-}))
		\jump{\varrho^k_{h}}^2_{\Gamma}(\vc{u}_{h}^{k} \cdot \nu)^+_h \\
		&\qquad \qquad\qquad\quad 
		-P''(\xi^\Gamma (\varrho^m_{-}, \varrho^m_{+}))
		\jump{\varrho^k_{h}}^2_{\Gamma}(\vc{u}_{h}^{k} \cdot \nu)^-_h\ dS(x)
		= \int_{\Omega}P(\varrho^{k-1})\ dx.
	\end{split}
\end{equation}
For $k=1,\ldots,M$ and $x\in\bigcup_{\Gamma \in \Gamma^I_{h}}\Gamma$, set
\begin{equation*}
	\vrho_{\dagger}^k(x):=
	\begin{cases}
		\max\{\vrho_{+}^k(x),\vrho_{-}^k(x)\}, & 1 < \gamma \leq 2, \\
		\min\{\vrho_{+}^k(x),\vrho_{-}^k(x)\}, & \gamma \ge 2,
	\end{cases}
\end{equation*}
and note that 
\begin{equation}\label{eq:doublederiv}
	\begin{split}
		&\Delta t\sum_{\Gamma \in \Gamma^I_{h}}
		\int_{\Gamma}P''(\xi^\Gamma (\vrho^m_{+},\vrho^m_{-}))
		\jump{\vrho^k_{h}}^2_{\Gamma}(\vc{u}_{h}^{k} \cdot \nu)_h^+ 
		\\ & \qquad\qquad\quad 
		-P''(\xi^\Gamma (\vrho^m_{-},\vrho^m_{+}))
		\jump{\vrho^k_{h}}^2_{\Gamma}(\vc{u}_{h}^{k} 
		\cdot \nu)^-_h\ dS(x) \\
		& \qquad
		\geq \Delta t \sum_{\Gamma \in
		\Gamma^I_{h}}\int_{\Gamma}P''(\vrho_{\dagger}^k)
		\jump{\vrho^k_{h}}^2_{\Gamma}
		\left((\vc{u}^k_{h} \cdot \nu)_h^+ 
		-(\vc{u}^k_{h} \cdot \nu)_h^-\right)\ dS(x).
	\end{split}
\end{equation}
Next, using $\vc{v}_{h} = \vc{u}^k_{h}$ as test 
function in \eqref{FEM:momentumeq}, we obtain the estimate
\begin{equation}\label{eq:stabilityeq2}
	\begin{split}
		\int_{\Om} p(\varrho_{h}^k)\Div_{h} \vc{u}_{h}^{k}\ dx 
		&= (\mu+\lambda)\|\Div_{h} \vc{u}_{h}^k\|_{L^2(\Om)}^2  
		+ \mu \|\Curl_h \vc{u}_{h}^k\|_{\vc{L}^2(\Om)}^2  
		- \int_{\Omega}\vc{f}_{h}^k\vc{u}^k_{h}\ dx \\
		&\qquad\qquad 
		+ \mu\sum_{\Gamma \in \Gamma^I_{h}}h^{\epsilon-1}
		\int_{\Gamma}\sjump{\vc{u}^k_{h}\cdot \nu}^2+ \sjump{\vc{u}^k_{h}\times \nu}^2\ dS(x) \\
		&\geq \mu\|\vc{u}_{h}^k\|_{\vc{V}_{h}}^2 
		-\int_{\Omega}\vc{f}_{h}^k\vc{u}^k_{h}\ dx, 
		\quad k=1, \dots, M.
	\end{split}
\end{equation}
Applying \eqref{eq:stabilityeq2} and \eqref{eq:doublederiv} to \eqref{eq:stabilityeq1} 
leads to the bound
\begin{equation*}
	\begin{split}
		&\int_{\Omega} P(\varrho_{h}^k)\ dx 
		+ \mu\Delta t \|\vc{u}_{h}^k\|_{\vc{V}_{h}(\Omega)}^2 
		+ \int_{\Omega} P''(\xi(\varrho_{h}^k,\varrho_{h}^{k-1}))
		\jump{\varrho^{k-1}_{h}}^2\ dx \\
		& \qquad \quad\quad
		+ \sum_{\Gamma \in \Gamma^I_{h}}\Delta t
		\int_{\Gamma}P''(\varrho^k_{\dagger})\jump{\varrho^k_{h}}_{\Gamma}^2
		\left((\vc{u}^k_{h} \cdot \nu)_h^+ -(\vc{u}^k_{h} \cdot \nu)_h^-\right)\ dS(x) \\
		& \quad \leq \int_{\Omega} P(\varrho_{h}^{k-1}) \ dx  
		+ \frac{1}{2\mu}\Delta t \int_{\Omega}|\vc{f}_{h}^k|^2dx 
		+\frac{\mu}{2}\Delta t \int_\Om |\vu_h^k|^2 \ dx.
	\end{split}
\end{equation*}
Summing over $k=1,\ldots,M$ 
yields \eqref{eq:stabilityincerrors}.
\end{proof} 

Since the stability estimate only provides the bound 
$p(\vrho_{h}) \in_{b} L^\infty(0,T;L^1(\Omega))$,
it is not clear that $p(\vrho_{h})$ converges 
weakly to an integrable function. Hence, we shall next 
establish that the pressure is in fact uniformly bounded 
in $L^2(0,T;L^2(\Omega))$.

To increase the readability we introduce the notation
$$
\avg{\phi} = \frac{1}{|\Om|}\int_\Om \phi \ dx.
$$


\begin{lemma}[Higher integrability on the pressure]\label{lemma:higherorderpressure} 
\solutiontext Then
$$
p(\varrho_{h}) \in_{b} 
L^{2}((0,T) \times \Omega).
$$
\end{lemma}

\begin{proof}
For $m=1,\ldots, M$, define $\vc{v}_h^m \in \vc{V}_h(\Om)$ by
$$
\vc{v}_{h}^m=\Pi_{h}^V 
\Bop{p(\varrho_{h}^m)-\avg{p(\varrho_{h}^m)}},
$$
where the operator $\Bop{\cdot}$ is defined in \eqref{def:bop}.	

Since $\Div \Pi_h^V = \Pi_h^Q \Div$, we have the identity 
$$
\Div_{h} \vc{v}_{h}^m 
= \Pi_{h}^Q \Div \Bop{p(\varrho_{h}^m)
-\avg{p(\varrho_{h}^m)}} 
=p(\varrho_{h}^m)-\avg{p(\varrho_{h}^m)}. 
$$

By using $\vc{v}_{h}^m$ as test function 
in the velocity method \eqref{FEM:momentumeq} and applying 
the previous identity, we obtain the relation
\begin{equation*}
	\begin{split}
		\int_{\Omega}p(\varrho_{h}^m)^2dx &= 
		|\Omega| \avg{p(\varrho_{h}^m)}^2 
		+  (\lambda + \mu) \int_{\Omega} \Div_{h} \vc{u}_{h}^m \Div_h \vc{v}_h^m dx \\
		&\quad +\int_{\Omega}\mu \Curl_{h} \vc{u}_{h}^m \Curl_{h} \vc{v}_{h}^m
		-\vc{f}_{h}^m \vc{v}_{h}^m dx \\
		&\quad+\mu\sum_{\Gamma \in \Gamma^I_{h}}h^{\epsilon-1}
		\int_{\Gamma}\sjump{\vc{u}^m_{h}\cdot \nu}\sjump{\vc{v}^m_{h}\cdot \nu}
		+\sjump{\vc{u}^m_{h} \times \nu}\sjump{\vc{v}^m_{h}\times \nu}\ dS(x).
	\end{split}
\end{equation*}

Repeated applications of H\"older's inequality yields
\begin{equation*}
	\begin{split}
		&\|p(\varrho_{h}^m)\|_{L^2(\Omega)}^2 \\
		&\quad \leq C\avg{p(\varrho_{h}^m)}^2
		+(\lambda + \mu)\|\Div_{h} \vc{u}_{h}^m\|_{L^2(\Om)}
		\|\Div_h \vc{v}_h^m\|_{L^2(\Om)} \\
		&\qquad + \mu\|\Curl_{h} \vc{u}_{h}^m\|_{\vc{L}^2(\Om)}
		\|\Curl_{h}\vc{v}^m_{h}\|_{\vc{L}^{2}(\Om)} 
		+ \|\vc{f}_{h}^m\|_{\vc{L}^2(\Omega)}
		\|\vc{v}_{h}^m\|_{\vc{L}^2(\Omega)} \\
		&\qquad
		+\mu\left|\sum_{\Gamma \in \Gamma^I_{h}}h^{\epsilon-1}
		\int_{\Gamma}\sjump{\vc{u}^m_{h}\cdot \nu}
		\sjump{\vc{v}^m_{h}\cdot \nu}
		+\sjump{\vc{u}^m_{h} \times \nu}
		\sjump{\vc{v}^m_{h}\times \nu}\ dS(x)\right|.
	\end{split}
\end{equation*}
To bound the jump terms we apply Lemma \ref{lemma:jumpcontrol}:
\begin{align*}
	&\|p(\varrho_{h}^m)\|_{L^2(\Omega)}^2 \\
	&\quad \leq C\Bigl[\avg{p(\varrho_{h}^m)}^2\\
	&\qquad \qquad + \left((1+h^\frac{\epsilon}{2})
	\|\vc{u}^m_{h}\|_{\vc{V}_{h}(\Om)}
	+\|\vc{f}_{h}^m\|_{\vc{L}^2(\Om)}\right)
	\|\Bop{p(\varrho_{h}^m) - \avg{p(\varrho_{h}^m)}}\|_{\vc{W}^{1,2}(\Om)}\Bigr] \\
	&\quad \leq C\left[\avg{p(\varrho_{h}^m)}^2+ 
	\Bigl((1+h^\frac{\epsilon}{2})\|\vc{u}^m_{h}\|_{\vc{V}_{h}(\Om)}
	+ \|\vc{f}_{h}^m\|_{\vc{L}^2(\Om)}\Bigr)
	\left(1+\|p(\varrho_{h}^m)\|_{L^2(\Om)}\right)\right],
\end{align*}
where the last inequality follows thanks to the estimate 
$\|\mathcal{B}[\phi]\|_{\vc{W}^{1,2}(\Omega)} 
\leq C\|\phi\|_{L^2(\Omega)}$.

Finally, an application of Cauchy's inequality (with $\epsilon$) yields
$$
\|p(\varrho_{h}^m)\|_{L^2(\Omega)}^2 
\leq C\left(1+ \avg{p(\varrho_{h}^m)}^2
+\|\vc{u}^m_{h}\|^2_{\vc{V}_{h}(\Om)}
+\|\vc{f}_{h}^m\|^2_{\vc{L}^2(\Om)}\right).
$$
Finally, we multiply this inequality by $\Delta t$, sum 
over $m=1,\ldots,M$, and apply Lemma \ref{lemma:stability}. 
This concludes the proof.
\end{proof}


\section{Convergence}\label{sec:conv}
Let $\Set{(\vr_h, \vu_h)}_{h>0}$ be a sequence of numerical solutions constructed according
to \eqref{eq:num-scheme-II} and Definition \ref{def:num-scheme}. In this section we will prove 
that a subsequence of this sequence converges to a weak solution of 
the semi--stationary Stokes system, thereby 
proving Theorem \ref{theorem:mainconvergence}.

In view of Section \ref{sec:basic}, we have the following 
$h$--independent bounds:
\begin{align*}
	& \vr_h \inb L^\infty(0,T;L^\gamma(\Om))
	\cap L^{2\gamma}((0,T)\times \Om),
	\\ & 
	\vu_h \inb L^2(0,T;\vc{L}^2(\Om)),
	\\ &
	\Div_h \vu_h \inb L^2(0,T;L^2(\Om)), 
	\\ &
	\Curl_h \vu_h \inb L^2(0,T;\vc{L}^2(\Om)).
\end{align*}
Consequently, we can assume that there exist limit functions 
$\vr \in L^\infty(0,T;L^\gamma(\Om))\cap L^{2\gamma}((0,T)\times \Om)$
and $\vu \in L^2(0,T;\vc{W}^{1,2}_0(\Om))$ such that
\begin{equation}\label{eq:conv1}
	\begin{split}
		&\vr_h \weakh \vr \quad \text{in $L^\infty(0,T;L^\gamma(\Om))
		\cap L^{2\gamma}((0,T)\times \Om)$,} \\
		&\vu_h \weakh \vu \quad \text{in $L^2(0,T;\vc{L}^2(\Om))$,}
	\end{split}
\end{equation}
and, by Lemma \ref{lemma:consistency},
\begin{equation}\label{eq:conv2}
	\begin{split}
		&\Div_h \vu_h \weakh \Div \vu \quad \text{in $L^2(0,T;L^2(\Om))$,} \\
		&\Curl_h \vu_h \weakh \Curl \vu \quad \text{in $L^2(0,T;\vc{L}^2(\Om))$.}
	\end{split}
\end{equation}
Moreover, 
\begin{equation*}
	\vrho_h^\gamma \overset{h\to 0}{\weak}\overline{\vrho^\gamma}, 
	\quad 
	\vrho_h^{\gamma+1} 
	\overset{h\to 0}{\weak}
	\overline{\vrho^{\gamma+1}}, 
	\quad
	\vrho_h\log\vrho_h \overset{h\to 0}{\weak} \overline{\vrho\log\vrho},
\end{equation*}
where each $\overset{h\to 0}\weak$ signifies weak convergence 
in a suitable $L^p$ space with $p>1$.

Finally, $\vrho_h$, $\vrho_h\log\vrho_h$ converge 
respectively to $\vrho$, $\overline{\vrho\log\vrho}$ 
in $C([0,T];L^p_{\text{weak}}(\Om))$ for some 
$1<p<\gamma$, cf.~Lemma \ref{lem:timecompactness} 
and also \cite{Feireisl:2004oe,Lions:1998ga}. 
In particular, $\vrho$, $\vrho\log \vrho$, and 
$\overline{\vrho\log\vrho}$ belong to $C([0,T];L^p_{\text{weak}}(\Om))$.

\subsection{Density method}

\begin{lemma}[Convergence of $\varrho \vc{u}$]\label{lem:convergenceofrhou}
Given \eqref{eq:conv1} and \eqref{eq:conv2},
$$
\vr_h \vu_h \weakh \vr \vu 
\quad \text{in the sense of distributions on $(0,T)\times \Om$.}
$$
\end{lemma}

\begin{proof}
Denote by $\widetilde{\vu_h}$ the function
$$
\widetilde{\vu_h}(t, \cdot) = \widetilde{\vu_h^m}, \quad t \in (t^{m-1},t^m],~m=1, \ldots, M,
$$
where $\widetilde{\vu_h^m}$ is defined in \eqref{eq:nedelec}.
By standard properties of the Nedelec interpolation operator, 
$\|\widetilde{\vu_h}\|_{L^2(0,T;\vc{L}^2(\Om))} 
\leq C\|\vu_h\|_{L^2(0,T;\vc{L}^2(\Om))}$. 
This, \eqref{eq:divagree}, and Lemma \ref{lemma:stability} 
allow us to conclude that $\widetilde{\vu_h} \inb L^2(0,T;\vc{W}^{\Div,2}(\Om))$ and 
$$
\sum_{m=1}^{M}\sum_{\Gamma \in \Gamma^I_{h}}\Delta t
\int_{\Gamma}P''(\varrho^m_{\dagger})
\jump{\varrho^m_{h}}_{\Gamma}^2
\left|\widetilde{\vc{u}^m_{h}} 
\cdot \nu\right| \ dS(x) \leq C.
$$

Using these bounds, we can apply to \eqref{FEM:oldcontequation} the calculations 
leading to Lemma 5.6 in \cite{Karlsen1}, resulting in the bound
\begin{equation}\label{eq:negbound}
	\Pth (\vr_h) \inb L^1(0,T;W^{-1,1}(\Om)).
\end{equation}
At the same time, Lemma \ref{lemma:compactembedding} tells us that
\begin{equation}\label{eq:translation}
	\|\vc{u}_{h}(t,x) - \vc{u}_{h}(t,x-\xi)\|_{L^2(0,T;\vc{L}^2(\Om_\xi))} 
	\rightarrow 0 \quad 
	\text{as $|\xi|\rightarrow 0$, uniformly in $h$.}
\end{equation}

In view of \eqref{eq:negbound} and \eqref{eq:translation}, an 
application of Lemma \ref{lemma:aubinlions} 
concludes the proof.
\end{proof}

\begin{lemma}[Continuity equation]\label{lemma:convergence-contequation}
The limit pair $(\vr, \vu)$ constructed in \eqref{eq:conv1} and \eqref{eq:conv2} 
is a weak solution of the continuity equation \eqref{eq:contequation} 
in the sense of Definition \ref{def:weak}.
\end{lemma}

\begin{proof}
Denote by $\widetilde{\vu_h}$ the function
$$
\widetilde{\vu_h}(t, \cdot) = \widetilde{\vu_h^m}, 
\quad t \in (t^{m-1},t^m], \quad m=1, \ldots, M,
$$
where $\widetilde{\vu_h^m}$ is defined in \eqref{eq:nedelec}. 

Fix a test function $\phi \in C_{0}^\infty([0,T)\times\cOm)$ and 
introduce the piecewise constant projections 
$\phi_{h}:=\Pi_{h}^Q \phi$, $\phi_{h}^m:= \Pi_{h}^Q \phi^m$, and
$\phi^m:=\frac{1}{\Delta t}\int_{t^{m-1}}^{t^m} \phi(t,\cdot)\ dt$.

By using $\phi_h^m$ as test function in \eqref{FEM:oldcontequation} and 
preforming the same calculations as in the proof of Lemma 6.4 in \cite{Karlsen1},
we work out the identity
\begin{equation}\label{eq:seherja}
	\begin{split}
		&\int_0^T\int_{\Omega} \Pth(\vr_h)\phi_{h}\ dxdt 
		= \int_0^T\int_{\Omega}
		\vrho_{h}\widetilde{\vc{u}_{h}} 
		\Grad\phi\ dxdt + \omega(h),
	\end{split}
\end{equation}
where $|\omega(h)| \leq Ch^\frac{1}{2}$ and 
\begin{equation*}
	\begin{split}
		\int_0^T\int_{\Omega} \Pth(\vr_h)\phi_{h}\ dxdt 
		\toh -\int_{0}^T\int_{\Omega}\vrho \phi_{t}\ dxdt 
		- \int_{\Omega}\vrho_{0}\phi(0,x)\ dx,
	\end{split}
\end{equation*}
where we have relied on \eqref{eq:conv1} and the strong 
convergence $\vr_h^0 \rightarrow \vr^0$ a.e.~in $\Om$.

Next,
$$
\int_0^T\int_{\Omega}\vrho_{h}\widetilde{\vc{u}_{h}} \Grad\phi \ dxdt
= \int_0^T\int_{\Omega}\vrho_{h}\vc{u}_{h} \Grad\phi
+\vr_h \left(\widetilde{\vu_h}-\vu_h \right)\Grad \phi \ dxdt.
$$
In view of Lemma \ref{lem:convergenceofrhou},
$$
\int_0^T \int_{\Omega} \vrho_{h}\vc{u}_{h} \Grad\phi \ dxdt
\toh \int_0^T\int_{\Omega} \vrho \vc{u}  \Grad \phi \ dxdt.
$$

By a standard error estimate for $\Pi_h^W$ (cf.~\cite{Nedelec:1980ec}),
\begin{align*}
	&\left|\int_0^T\int_\Om\vr_h 
	\left(\widetilde{\vu_h}-\vu_h \right)\Grad \phi~dxdt\right|  \\
	&\quad \leq hC \|\vr_h\|_{L^2(0,T;L^2(\Om))}
	\|\Grad_h \vu_h\|_{L^2(0,T;L^2(\Om))}
	\|\Grad \phi\|_{L^\infty(0,T;L^\infty(\Om))}
	\leq C h^{\frac{2-\epsilon}{4}},
\end{align*}
where the final inequality follows 
from Lemmas \ref{lem:blowup}, \ref{lemma:stability}, 
and \ref{lemma:higherorderpressure}.

Summarizing, sending $h \rightarrow 0$ in \eqref{eq:seherja} delivers 
the desired result \eqref{eq:weak-rho}.
\end{proof}

\subsection{Strong convergence of density approximations}\label{sec:strong-conv-of-vrho}
To establish the strong convergence of the density approximations $\vr_h$, 
we will utilize a weak continuity property of the effective viscous flux:
$
\eff(\vrho_{h},\vc{u}_{h})  = p(\vr_h) - (\lambda + \mu)\Div \vu_h.
$

To derive this property we exploit the div--curl structure of the velocity scheme \eqref{FEM:momentumeq}
combined with the commutative properties \eqref{eq:commute} of $\vc{V}_h$.
More specifically, in view of the commutative property \eqref{eq:commute},
the function $\vc{v}_h = \Pi_h^V\Grad \Delta^{-1}\vr_h$ satisfies
$\Div_h \vc{v}_h = \vr_h$ and $\Curl_h \vc{v}_h = 0$ on elements away from 
the boundary. The crucial point is that the curl part 
of the velocity method \eqref{FEM:momentumeq} vanishes 
when this $\vc{v}_h$ is utilized as a test function.

\begin{lemma}[Discrete effective viscous flux] \label{lemma:effectiveflux} 
Given the weak convergences listed in \eqref{eq:conv1} and \eqref{eq:conv2},
$$
\lim_{h \to 0}\int_{0}^T\int_{\Omega}
\eff(\vrho_{h},\vc{u}_{h})\, \vrho_{h}\, \phi\psi \ dxds 
=\int_{0}^T\int_{\Omega} 
\overline{\eff(\vrho,\vc{u})}\,\vrho\, \phi\psi \ dxds, 
$$
for all $\phi \in C_0^\infty(\Om)$ and $\psi \in C^\infty(0,T)$.
\end{lemma}

\begin{proof}
Fix $\phi \in C_{0}^\infty(\Omega)$, $\psi \in C^\infty(0,T)$, and for each $h>0$ 
introduce the test function
$$
\vc{v}_{h}(\cdot,t)= \psi\Pi_{h}^V 
\left[\phi \Aop{\varrho_{h}-\varrho}(\cdot,t)\right], 
\qquad t \in (0,T).
$$
where the operator $\Aop{\cdot}$ is defined in \eqref{def:aop}.	

By virtue of \eqref{eq:commute} 
and $\Curl \Aop{\cdot} = 0$, we have the identities
\begin{align*}
	\Div_{h} \vc{v}_{h} 
	= \psi\Pi_{h}^Q\left(\Grad\phi \Aop{\varrho_{h} - \varrho}\right)
	+ \psi\Pi_{h}^Q \left(\phi(\varrho_{h}-\varrho)\right)
\end{align*}
and
\begin{equation*}
	\Curl_{h} \vc{v}_{h}=\psi\Pi_{h}^Q 
	\left( \Grad\phi \times 
	\gradlaplaceinv{\varrho_{h}-\varrho}\right).
\end{equation*}

For $m=1, \ldots, M$, set 
$\vc{v}_{h}^m := \frac{1}{\Delta t}\int_{t^{m-1}}^{t^m} \vc{v}_{h}(\cdot, s)\ ds$.
Taking $\vc{v}_{h}^m$ as test function in the velocity method \eqref{FEM:momentumeq},
utilizing the above identities, multiplying by $\Delta t$, and summing over $m$,
we obtain
\begin{equation} \label{eq:viscconv1}
	\begin{split}
		&\int_{0}^{T}\int_{\Om} \eff(\vr_h, \vu_h)(\varrho_{h} - \varrho)\phi\psi\ dxds 
		\\ &
		= -\int_{0}^{T}\int_{\Om} \eff(\vr_h, \vu_h)\Grad \phi 
		\cdot \Aop{\varrho_{h}-\varrho}\psi
		+ \vc{f}_{h}\left(\Pi_{h}^V \phi \Aop{\varrho_{h}-\varrho }\right)\psi \ dxds 
		\\ & \qquad 
		+\int_0^T\int_\Om  \mu\Curl_{h} \vc{u}_{h} 
		\left(\Grad\phi \times \Aop{\varrho_{h} - \varrho}\right)\psi~ dxds \\
		& \qquad  + \mu\sum_{\Gamma \in \Gamma^I_{h}}h^{\epsilon-1}
		\int_{0}^T\psi\int_{\Gamma}\sjump{\vc{u}_{h} \cdot \nu}
		\sjump{\vc{v}_{h} \cdot \nu}
		+\sjump{\vc{u}_{h} \times \nu}
		\sjump{\vc{v}_{h} \times \nu} \ dS(x)ds.
	\end{split}
\end{equation}

In view of \eqref{eq:negbound}, the following $h$--independent bounds are immediate:
\begin{equation}\label{eq:abound}
	\begin{split}
		& \Pth\Aop{\vrho_{h}} = \Aop{\Pth \vr_h} \inb L^1(0,T;W^{-1,1}(\Om)),\\
		& \Aop{\vrho_{h}} \inb L^2(0,T;W^{1,2}(\Om)).
	\end{split}
\end{equation}
Consequently, Lemma \ref{lemma:aubinlions} can be applied with the result that 
$\left(\Aop{\vr_h}\right)^2 \overset{h \rightarrow 0}{\weak} 
\bigl(\Aop{\vr}\bigr)^2$. Thus, 
\begin{equation}\label{eq:aconv}
	\Aop{\vrho_{h} - \vrho} \overset{h \rightarrow 0}{\rightarrow} 0,
	\quad \text{in }L^2(0,T;\vc{L}^2(\Om)).
\end{equation}
Now, using \eqref{eq:aconv} together with \eqref{eq:conv1} and \eqref{eq:conv2},  
we send $h \rightarrow 0$ in  \eqref{eq:viscconv1} to obtain
\begin{equation}\label{eq:onlyfluxleft}
	\begin{split}
		&\lim_{h \rightarrow 0}\int_{0}^{T}\int_{\Om}\eff(\vr_h, \vu_h)(\varrho_{h} 
		- \varrho)\phi\psi\ dxds \\
		&=\lim_{h \rightarrow 0} \mu\sum_{\Gamma \in 
		\Gamma^I_{h}}h^{\epsilon -1}\int_{0}^T\psi\int_{\Gamma}
		\jump{\vc{u}_{h} \cdot \nu}\jump{\Pi_{h}^V (\phi \Aop{\varrho_{h}
		-\varrho})\cdot \nu} \\
		& \qquad\qquad \qquad \qquad \qquad \qquad
		+\jump{\vc{u}_{h} \times \nu}\jump{\Pi_{h}^V 
		(\phi \Aop{\varrho_{h}- \varrho})\times \nu} \ dS(x)dt.
	\end{split}
\end{equation}

Lemma \ref{lemma:jumpcontrol} yields
\begin{equation*}
	\begin{split}
		& \left|\mu\sum_{\Gamma \in \Gamma^I_{h}}h^{\epsilon -1}\int_{0}^T\psi\int_{\Gamma}
		\sjump{\vc{u}_{h} \cdot \nu}\sjump{\Pi_{h}^V (\phi \Aop{\varrho_{h}
		-\varrho})\cdot \nu} \ dS(x)ds\right| \\
		&\qquad \quad 
		+\left|\mu\sum_{\Gamma \in \Gamma^I_{h}}
		h^{\epsilon -1}\int_{0}^T\psi\int_{\Gamma}
		\sjump{\vc{u}_{h} \times \nu}\sjump{\Pi_{h}^V (\phi \Aop{\varrho_{h}
		-\varrho})\times \nu} \ dS(x)ds \right| \\
		&\leq h^\frac{\epsilon}{2}C\|\psi\|_{L^\infty(0,T)}\|\vc{u}_{h}\|_{L^2(0,T;\vc{V}_{h}(\Om))}
		\|\Grad (\phi\Aop{\varrho_{h} - \varrho})\|_{L^2(0,T;\vc{L}^2(\Om))} 
		\leq Ch^\frac{\epsilon}{2},
	\end{split}
\end{equation*}
where the last inequality follows from \eqref{eq:abound} and Lemmas \ref{lemma:stability} and \ref{lemma:higherorderpressure}. 
Applying the previous bound to \eqref{eq:onlyfluxleft} yields the desired result.
\end{proof}

We can now infer the strong convergence of the density approximations.

\begin{lemma}[Strong convergence of $\vrho_h$]\label{lem:strong-conv}
Suppose that \eqref{eq:conv1}--\eqref{eq:conv2} holds. 
Then, passing to a subsequence as $h\to 0$ if necessary,
$$
\vrho_{h} \rightarrow \vrho\quad 
\text{a.e.~in~$(0,T)\times \Omega$.}
$$
\end{lemma}

\begin{proof}	
In view of Lemma \ref{lemma:convergence-contequation}, the 
limit $(\vrho,\vc{u})$ is a weak solution 
of the continuity equation and hence, by Lemma \ref{lemma:feireisl}, 
also a renormalized solution:
$$
\left(\vrho\log \vrho\right)_t 
+ \Div \left( \left(\vrho\log\vrho\right)
\vc{u}\right)=\vrho \Div \vc{u} 
\quad \text{in the weak sense on $\cDom$.}
$$

Since $t\mapsto \vrho\log \vrho$ is continuous 
with values in an appropriate Lebesgue space 
equipped with the weak topology, we can use this 
equation to obtain for any $t>0$
\begin{equation}\label{eq:stngdenconv-eq1}
	\int_{\Omega} \left(\vrho \log \vrho\right)(t)\ dx
	-\int_{\Omega}\vrho_{0}\log \vrho_{0}\ dx
	= -\int_{0}^t \int_{\Omega}\vrho \Div \vc{u}\ dxds
\end{equation}

Next, we specify $\phi_h\equiv 1$ as test function in the 
renormalized scheme \eqref{FEM:renormalized}, multiply by $\Delta t$,
and sum the result over $m$. Making use of the 
convexity of $z\log z$, we conclude that for any $m=1,\dots,M$ 
\begin{equation}\label{eq:stngdenconv-eq2}
	\int_{\Omega}\vrho^m_{h}\log \vrho^m_{h}\ dx
	-\int_{\Omega}\vrho^{0}_h\log \vrho^{0}_h\ dx  
	\leq -\sum_{k=1}^m \Delta t\int_{\Omega}\vrho^m_{h}
	\Div \vc{u}^m_{h}\ dxdt.
\end{equation}

In view of the convergences stated at the beginning of this section 
and strong convergence of the initial data, we 
can send $h \to 0$ in \eqref{eq:stngdenconv-eq2} to obtain
\begin{equation}\label{eq:stngdenconv-eq3}
	\int_{\Omega} \Bigl(\overline{\vrho \log \vrho}\Bigr)(t)\ dx
	-\int_{\Omega}\vrho_{0}\log \vrho_{0}\ dx
	\le -\int_{0}^t \int_{\Omega}\overline{\vrho \Div \vc{u}}\ dxds.
\end{equation}

Subtracting \eqref{eq:stngdenconv-eq1} 
from \eqref{eq:stngdenconv-eq3} gives
\begin{align*}
	\int_{\Omega}\Bigl(\overline{\vrho \log \vrho}-\vrho \log \vrho\Bigr)(t)\ dx
	& \leq -\int_{0}^t\int_{\Omega}
	\overline{\vrho\Div \vc{u}}-\vrho \Div \vc{u}\ dxds,
\end{align*}
for any $t\in (0,T)$. Lemma \ref{lemma:effectiveflux} tells us that
\begin{equation*}
	\int_{0}^t\int_{\Omega} \left(\overline{\vrho \Div \vc{u}}
	-\vrho \Div \vc{u}\right)\phi  ~dxds 
	= \frac{a}{\mu + \lambda}\int_{0}^t\int_{\Omega}
	\left(\overline{\vrho^{\gamma +1}}
	-\overline{\vrho^\gamma} \vrho\right)\phi~ dxds\ge 0,
\end{equation*}
for all $\phi\in C_0^\infty(\Om)\cap \Set{\phi \geq 0}$, where the last inequality follows 
as in \cite{Feireisl:2004oe,Lions:1998ga}, so 
the following relation holds: 
$$
\overline{\vrho \log \vrho}=\vrho \log \vrho
\quad \text{a.e.~in $\Dom$.}
$$
Now an application of Lemma \ref{lem:prelim} finishes the proof.
\end{proof}

\subsection{Velocity method}


\begin{lemma}[Velocity equation]
The limit pair $(\vr, \vu)$ constructed in \eqref{eq:conv1}--\eqref{eq:conv2}
is a weak solution to the velocity equation \eqref{eq:momentumeq} in the
sense of Definition \ref{def:weak}.
\end{lemma}

\begin{proof}
Fix $\vc{v} \in L^2(0,T;\vc{W}^{1,2}_0(\Om))$, and set 
$\vc{v}_h = \Pi_h^V \vc{v}$ and 
$\vc{v}_h^m = \frac{1}{\Delta t}
\int_{t^{m-1}}^{t^m}\vc{v}_h\ dt$.

Then, setting $\vc{v}^m_h$ as test function in the velocity method \eqref{FEM:momentumeq}, 
multiplying with $\Delta t$, and summing over all $m=1,\ldots,M$, leads
to the identity
\begin{equation}\label{eq:lastinpaper}
	\begin{split}
		&\int_0^T\int_\Om \mu \Curl_h \vu_h \Curl \vc{v} 
		+ \left[(\mu +\lambda) \Div_h \vu_h - p(\vr_h) \right]\Div \vc{v} \ dxdt \\
		&\quad \quad + \mu
		\sum_{\Gamma \in \Gamma^I_{h}}h^{\epsilon-1}\int_0^T\int_{\Gamma}
		\sjump{\vc{u}_{h}\cdot \nu}\sjump{\left(\Pi_{h}^V\vc{v}\right)\cdot \nu} \\
		&\qquad \qquad \qquad \qquad \qquad  
		+ \sjump{\vc{u}_{h}\times \nu}\sjump{\left(\Pi_{h}^V\vc{v}\right)\times \nu} \ dS(x)dt	 \\
		&= \int_0^T\int_\Om \vc{f}_h \Pi_h^V\vc{v} \ dxdt,
	\end{split}
\end{equation}
where we have also used \eqref{eq:commute}. From Lemma \ref{lem:strong-conv} and \eqref{eq:conv1},
we have that $p(\vr_h) \overset{h \rightarrow 0}{\rightarrow} p(\vr)$ in $L^2(0,T;L^2(\Om))$.
Furthermore,  Lemma \ref{lemma:jumpcontrol} tells us that the jump terms 
converge to zero. Hence, we can send $h \rightarrow 0$ in \eqref{eq:lastinpaper}
to obtain that the limit $(\vr, \vu)$ constructed in \eqref{eq:conv1}--\eqref{eq:conv2}
satisfies \eqref{eq:weak-u} for all test functions $\vc{v} \in L^2(0,T;\vc{W}^{1,2}_0(\Om))$.
\end{proof}


\end{document}